\documentclass[final,leqno]{siamltex704}
\usepackage{amsmath}
\usepackage{graphicx}
\usepackage[notcite,notref]{showkeys}
\usepackage{mathrsfs}
\usepackage{bbm}
\usepackage{bm}
\usepackage{float}
\usepackage{cases}
\usepackage{amsfonts,amssymb}
\usepackage{dsfont}
\usepackage{pifont}
\usepackage{tikz}
\usepackage{wrapfig} 
\usepackage{hyperref}
\usepackage{multirow}
\usepackage{lineno}
\usepackage{mathtools}
\usepackage{subfigure}
\usepackage{color}
\usepackage{dashrule}
\usepackage{amsmath}

\numberwithin{equation}{section}
\def\3bar{{|\hspace{-.02in}|\hspace{-.02in}|}}
\def\E{{\mathcal{E}}}
\def\T{{\mathcal{T}}}

\def\pT{{\partial T}}

\def\bn{{\mathbf{n}}}

\def\ljump{{[\![}}
\def\rjump{{]\!]}}

\def\argmin{\operatornamewithlimits{arg\ min}}

\title {Primal-Dual Weak Galerkin Finite Element Methods for Elliptic Cauchy Problems}

\author{Chunmei Wang \thanks{Department of Mathematics, Texas State University, San Marcos, TX 78666, USA. The research of Chunmei Wang was partially supported by National Science Foundation Awards DMS-1648171 and DMS-1749707.} \and Junping Wang\thanks{Division of Mathematical Sciences, National Science Foundation, Alexandria, VA 22314 (jwang@nsf.gov). The research of Junping Wang was supported by the NSF IR/D program, while working at National Science Foundation.
However, any opinion, finding, and conclusions or recommendations
expressed in this material are those of the author and do not
necessarily reflect the views of the National Science Foundation.}}

\begin{document}
\maketitle

\begin{abstract}
The authors propose and analyze a well-posed numerical scheme for a type of ill-posed elliptic Cauchy problem by using a constrained minimization approach combined with the weak Galerkin finite element method. The resulting Euler-Lagrange formulation yields a system of equations involving the original equation for the primal variable and its adjoint for the dual variable, and is thus an example of the primal-dual weak Galerkin finite element method. This new primal-dual weak Galerkin algorithm is consistent in the sense that the system is symmetric, well-posed, and is satisfied by the exact solution. A certain stability and error estimates were derived in discrete Sobolev norms,
including one in a weak $L^2$ topology. Some numerical results are reported to illustrate and validate the theory developed in the paper.
\end{abstract}

\begin{keywords} primal-dual weak Galerkin, finite element methods, elliptic Cauchy problem.
\end{keywords}

\begin{AMS}
65N30, 65N15, 65N12, 65N20
\end{AMS}

\pagestyle{myheadings}

\section{Introduction}

This paper is concerned with the development of new numerical methods for a type of Cauchy problems for the second order elliptic equation. For simplicity, we consider a model elliptic Cauchy problem that seeks
an unknown function $u=u(x)$ satisfying
\begin{equation}\label{EQ:model-problem}
\begin{split}
\Delta  u  =&f,\quad \ \text{in}\quad
\Omega,\\
u =& g_1,\quad \text{on}\quad   \Gamma_d,\\
\partial_\bn u =& g_2,\quad \text{on}\quad \Gamma_n,
\end{split}
\end{equation}
where $\Omega$ is an open bounded domain in $\mathbb R^d$ $(d=2, 3)$
with Lipschitz continuous boundary $\partial\Omega$; $\Gamma_d$ and $\Gamma_n$ are two segments of the domain boundary $\partial\Omega$; $f\in L^2(\Omega)$, the Cauchy data $g_1$ and $g_2$ are two given functions defined on the appropriate parts of the boundary. $\partial_\bn u$ stands for the directional derivative of $u=u(x)$ in the outward normal direction $\bn$ on $\partial\Omega$. The elliptic Cauchy model problem (\ref{EQ:model-problem}) consists of solving a PDE on a domain where over-specified boundary conditions are given on parts of its boundary, which can be interpreted as solving a data completion problem with missing boundary conditions on the remaining parts of the domain boundary.

The study of the elliptic Cauchy problem \eqref{EQ:model-problem} has a long history tracing back to Hadamard \cite{e17, k22, h1, h2, h3}, where a basis for the notation of a type of well-posed problems was laid out. Hadamard used the problem \eqref{EQ:model-problem} with $\Gamma_d=\Gamma_n$ to demonstrate the ill-posedness of the problem by constructing an example for which the solution does not depend continuously on the Cauchy data. The work of Hadamard and others indicates that a small perturbation or error in the data may lead to an enormous error in the numerical solution for elliptic Cauchy problems \cite{s13, s14, s15}. As shown by the Schwartz reflection principle \cite{Gilbarg-Trudinger}, the existence of solutions for arbitrary Cauchy data $g_1$ and $g_2$ is generally not guaranteed for the problem \eqref{EQ:model-problem}. But it has been shown in \cite{Andrieux} that there exists a dense subset $M$ of $H^{\frac12}(\Gamma_d)\times [H_{00}^{\frac12}(\Gamma_n)]^\prime$ such that the problem \eqref{EQ:model-problem} has a solution $u\in H^1(\Omega)$ for any Cauchy data $g_1\times g_2\in M$. It is well-known that the solution of the problem \eqref{EQ:model-problem}, if it exists, must be unique, provided that $\Gamma_d\cap\Gamma_n$ is a nontrivial portion of the domain boundary. Throughout this paper, we assume that the Cauchy data is compatible such that the solution exists; furthermore, we assume that $\Gamma_d\cap\Gamma_n$ is a nontrivial portion of the domain boundary so that the solution of the elliptic Cauchy problem (\ref{EQ:model-problem}) is unique.

The elliptic Cauchy problems arise in many areas of science and engineering, such as wave propagation, vibration, electromagnetic scattering, geophysics, cardiology, steady-state inverse heat conduction, and nondestructive testing. It is widely recognized that the Cauchy problem for Laplace's equation, and more generally for second order elliptic equations, plays a critical role in many inverse boundary value problems modeled by elliptic partial differential equations. Among some popular examples, we mention the problem arising in electrostatic or thermal imagining methods in nondestructive testing and evaluations \cite{cakoni}. In this application, the function $u=u(x)$ can be understood as the electrostatic potential in a conducting body occupying domain $\Omega$ of which only the portion $\Gamma_d=\Gamma_n$ of the boundary is accessible to measurements. The PDE model involves the Laplace equation with no source term (i.e., $f=0$). The goal in this application is to determine the shape of the inaccessible portion of the boundary from the imposed voltage $u|_{\Gamma_d}$ and the measured current (flux) $\partial_\bn u$ on $\Gamma_n$. Readers are referred to \cite{s1, s2, s3, s4, s5, s6, s7, s8, s9, s10, s11, s12, k6, k10, k11} and the references cited therein for more examples and results on elliptic Cauchy problems.

There have been some numerical methods in the literature for approximating elliptic Cauchy problems based on two strategies: (1) reformulate the problem as an equation with missing boundary data for which Tikhonov regularization is applied for a determination of the solution; (2) approximate the ill-posed problem iteratively by a sequence of well-posed problems with the same equation. In both strategies, approximations of harmonic functions, and, particularly, their boundary values, are computed by using carefully designed numerical schemes. In \cite{s16}, a numerical method for the Cauchy problem of the Laplace equation was devised based on boundary integral equations through the use of the single-layer potential function and jump relations. In \cite{s17, s18}, the authors developed a moment method and a boundary particle method. In \cite{s4, s19, s20, s21,s22,s12}, the authors proposed and analyzed several numerical methods, including the alternating iterative boundary element method, the conjugate gradient boundary element method, the boundary knot method, and the method of fundamental solutions for the elliptic Cauchy problem. In \cite{Leitao}, the authors developed two methods of level set type for solving the elliptic Cauchy problem. In \cite{Falk-Monk}, the authors introduced an optimization approach based on least squares and Tikhonov regularization techniques. A finite element method, based on an optimal control characterization of the Cauchy problem, was introduced and analyzed in \cite{r10}. In \cite{Burman, ErikBurman-EllipticCauchy}, the author developed a stabilized finite element procedure based on a general framework involving both the original equation and its adjoint. The numerical approach proposed in
\cite{Burman, ErikBurman-EllipticCauchy} is applicable to a wide class of ill-posed problems for which only weak continuous dependence is necessary. It should be noted that a variety of theoretical and applied work have also been developed for the elliptic Cauchy problem by using the Steklov-Poincare theory \cite{r2, r3, Belgacem}, regularization methods \cite{r5, r6}, quasi-reversibility method \cite{Bourgeois} or minimal error methods \cite{r8, r9}.

The goal of this paper is to devise a new numerical scheme with rigorous mathematical convergence for the elliptic Cauchy problem \eqref{EQ:model-problem} by using a newly-developed primal-dual weak Galerkin (PD-WG) finite element method  \cite{ww2016,ww2017}. The key to PD-WG is the determination of approximate solutions from the space of weak finite element functions, with the least discontinuity requirement across the boundary of each element while (weakly) satisfying the original differential equation on each element. In the following paragraph we sketch the main idea behind the PD-WG finite element method for the elliptic Cauchy problem \eqref{EQ:model-problem}, with notations to be consistent with the rest of the paper.

A function $u\in H^{1}(\Omega)$ is said to be a weak solution of the elliptic Cauchy problem \eqref{EQ:model-problem} if the following are satisfied: (i) $u|_{\Gamma_d} = g_1 \in H^{\frac12}(\Gamma_d)$, (ii) $\partial_\bn u|_{\Gamma_n}=g_2|_{\Gamma_n} \in [H_{00}^{\frac12}(\Gamma_n)]'$, and (iii) on any  control element $T\subset \Omega$, one has
\begin{equation}\label{weakform-01}
(u, \Delta v)_T - \langle u, \partial_\bn v\rangle_\pT + \langle \partial_\bn u, v\rangle_{\pT} =(f,v)_T, \qquad \forall v\in H^2(T),
\end{equation}
where $\langle\cdot, \cdot\rangle_\pT$ stands for the pairing between $H^{\frac12}(\pT)$ and $H^{-\frac12}(\pT)$. The left-hand side of \eqref{weakform-01} defines the weak Laplacian operator $\Delta_w$ introduced originally in \cite{mwy0927,ww} so that \eqref{weakform-01} can be rewritten as
\begin{equation}\label{weakform-02}
\langle\Delta_w \{u\}, v\rangle_T  =(f,v)_T, \qquad \forall v\in H^2(T),
\end{equation}
where $\{u\} = \{u|_T, u|_\pT, \bn (\partial_\bn u |_\pT)\}$ is a weak function on $T$ (see \eqref{2.1} or \cite{mwy0927,ww} for definition). In the weak Galerkin context, the space of weak functions is approximated by weak finite element space consisting of piecewise polynomials (often without any continuity requirement). The weak Laplacian operator $\Delta_w$ is correspondingly approximated by a discrete analogue denoted as $\Delta_{w,h}$ (see the definition of $\Delta_{w, r, K}$ in Section 2 for its precise definition and computation) so that the equation (\ref{EQ:model-problem}) can be discretized by
\begin{equation}\label{EQ:10-12-2015:01}
(\Delta_{w,h} u_h, w) = (f, w), \qquad \forall w\in W_{h},
\end{equation}
where $W_{h}$ is a test space and $V^g_{h}$ is the trial space consisting of weak finite element functions with proper boundary values. The discrete problem (\ref{EQ:10-12-2015:01}), however, is not well-posed unless an {\em inf-sup} condition of Babu\u{s}ka \cite{babuska} and Brezzi \cite{b1974} is satisfied. The primal-dual formulation was designed to overcome this difficulty through the use of a constrained minimization formulation which seeks $u_h\in V^g_{h}$ as a
minimizer of a prescribed non-negative quadratic functional
$J(v)=\frac12 s(v,v)$ with constraint given by the equation
(\ref{EQ:10-12-2015:01}). The functional $J(v)$ measures the level of
``continuity" of $v\in V^g_{h}$ in the sense that $v\in V^g_{h}$ is
a classical $C^1$-conforming element if and only if $s(v,v)=0$. The resulting Euler-Lagrange equation for this constraint minimization problem gives rise to a symmetric numerical algorithm involving not only the original unknown function (primal variable) $u_h$, but also a dual variable $\lambda_h$. A formal description of the scheme can then be given as follows: Find
$u_h\in V_{h}^g$ and $\lambda_h\in W_{h}$ such that
\begin{equation}\label{primal-dual-wg}
\begin{split}
s(u_h, v) + (\Delta_{w,h} v,\lambda_h) &=0,\qquad\qquad
\forall v \in V_{h}^0,\\
(\Delta_{w,h} u_h, w) &= (f, w) ,\qquad\forall
 w \in W_{h},
 \end{split}
\end{equation}
where $s(\cdot,\cdot)$ is a bilinear form in the finite element space $V_{h}$ known as the {\em stabilizer} or {\em smoother} that enforces certain weak continuity for the approximation $u_h$. Numerical schemes in the form of (\ref{primal-dual-wg}) have been named
{\em primal-dual weak Galerkin finite element methods} in
\cite{ww2016, ww2017}, and they were known as {\em stabilized finite
element methods} in \cite{Burman, ErikBurman-EllipticCauchy-SIAM02, ErikBurman-EllipticCauchy} in different finite element contexts.

Our primal-dual weak Galerkin algorithm \eqref{primal-dual-wg} has the following advantages over the existing schemes: (1) it offers a symmetric and well-posed problem for the ill-posed elliptic Cauchy problem, (2) it is consistent in the sense that the system is satisfied by the exact solution (if it exists), and (3) PD-WG works well for a wide class of PDE problems for which no traditional variational formulations are available. In addition, like other weak Galerkin and discontinuous Galerkin finite element methods, the numerical algorithms arising from PD-WG admit general finite element partitions consisting of arbitrary polygons or polyhedra.


Throughout the paper, we follow the usual notation for Sobolev spaces and norms. For any open bounded domain $D\subset \mathbb{R}^d$ ($d$-dimensional Euclidean space) with Lipschitz continuous boundary, we use $\|\cdot\|_{s,D}$ and
$|\cdot|_{s,D}$ to denote the norm and seminorm in the Sobolev
space $H^s(D)$ for any $s\ge 0$, respectively. The inner product in
$H^s(D)$ is denoted by $(\cdot,\cdot)_{s,D}$. The space $H^0(D)$
coincides with $L^2(D)$, for which the norm and the inner product
are denoted by $\|\cdot \|_{D}$ and $(\cdot,\cdot)_{D}$,
respectively. When $D=\Omega$, we shall drop the subscript $D$ in
the norm and inner product notation. For convenience, throughout the
paper, we use ``$\lesssim$ '' to denote ``less than or equal to up
to a general constant independent of the mesh size or functions
appearing in the inequality".

The paper is organized as follows. Section \ref{Section:Hessian} is devoted to a discussion of the weak Laplacian operator as well as its discretization. In
Section \ref{Section:WGFEM}, we give a detailed description of the primal-dual weak Galerkin algorithm for the elliptic Cauchy problem (\ref{EQ:model-problem}). Section \ref{Section:stability} is devoted to the presentation of some
technical results, including the critical {\em inf-sup} condition. In Section \ref{Section:convergence}, we establish some convergence results based on the derivation of an error equation. In Section \ref{Section:L2Error}, an optimal order of error estimate is derived for the primal-dual WG finite element
approximations in a weak $L^2$ topology. Finally in Section \ref{Section:NE}, we report a series of numerical results that demonstrate the effectiveness and accuracy of the theory developed in the previous sections.

\section{Weak Laplacian and Discrete Weak Laplacian}\label{Section:Hessian}

Let $K$ be a polygonal or polyhedral element with boundary $\partial
K$. A weak function on $K$ refers to a triplet $v=\{v_0,
v_b, v_n \bn\}$ such that $v_0\in L^2(K)$, $v_b\in L^{2}(\partial
K)$ and $v_n\in L^{2}(\partial K)$. Here $\bn$ is the outward normal
direction on $\partial K$. The first component $v_0$ represents the
``value" of $v$ in the interior of $K$, and the rest, namely $v_b$ and $v_n$, are reserved for the boundary information of $v$. In application to the Laplacian operator, $v_b$ denotes the boundary value of $v$ and $v_n$ is the outward normal derivative of $v$ on $\partial K$; i.e., $v_n \approx \nabla v\cdot \bn$. In general, $v_b$ and $v_n$ are assumed to be independent of the trace of $v_0$ and
$\nabla v_0 \cdot \bn $, respectively, on $\partial K$, but the special cases of $v_b= v_0|_{\partial K}$ and $v_n= (\nabla v_0 \cdot \bn)|_{\partial K}$ are completely legitimate, and when this happens, the function $v=\{v_0, v_b, v_n \bn\}$ is uniquely determined by $v_0$ and shall be simply denoted as $v=v_0$.

Denote by $W(K)$ the space of all weak functions on $K$; i.e.,
\begin{equation}\label{2.1}
W(K)=\{v=\{v_0,v_b, v_n\bn\}: v_0\in L^2(K), v_b\in L^{2}(\partial
K), v_n\in L^{2}(\partial K) \}.
\end{equation}

The weak Laplacian, denoted by $\Delta_{w}$, is a linear
operator from $W(K)$ to the dual of $H^{2}(K)$ such that for any
$v\in W(K)$, $\Delta_w v$ is a bounded linear functional on $H^2(K)$
defined by
\begin{equation}\label{2.3}
 \langle\Delta _{ w}v,\varphi\rangle_K=(v_0,\Delta \varphi)_K-
 \langle v_b,\nabla \varphi\cdot \bn \rangle_{\partial K}+
 \langle v_n,\varphi  \rangle_{\partial K},\quad \forall \varphi\in H^2(T),
 \end{equation}
where the left-hand side of (\ref{2.3})
represents the action of the linear functional $\Delta_w v$ on
$\varphi\in H^2(K)$.

For any non-negative integer $r\ge 0$, let $P_r(K)$ be the space of
polynomials on $K$ with total degree $r$ and less. A discrete weak
Laplacian on $K$, denoted by $\Delta_{w,r,K}$, is a linear operator
from $W(K)$ to $P_r(K)$ such that for any $v\in W(K)$,
$\Delta_{w,r,K}v$ is the unique polynomial in $P_r(K)$ satisfying
\begin{equation}\label{2.4}
 (\Delta _{w,r,K} v,\varphi)_K=(v_0,\Delta \varphi)_K-
 \langle v_b,\nabla \varphi\cdot \bn \rangle_{\partial K}+
 \langle v_n,\varphi  \rangle_{\partial K},\quad \forall \varphi \in P_r(K).
 \end{equation}
For smooth $v_0\in
H^2(K)$, one may apply the usual integration by parts to the first
term on the right-hand side of (\ref{2.4}) to obtain
  \begin{equation}\label{2.4new}
 (\Delta _{w,r,K} v,\varphi)_K=(\Delta v_0,  \varphi)_K+
 \langle v_0- v_b,\nabla \varphi\cdot \bn \rangle_{\partial K}-
 \langle  \nabla v_0\cdot \bn - v_n,\varphi  \rangle_{\partial K}.
 \end{equation}
In particular, if $v_b=v_0$ and $v_n=\nabla v_0\cdot \bn$ on
$\partial K$, we have
\begin{equation}\label{2.4new2}
 (\Delta_{w,r,K} v,\varphi)_K=(\Delta v_0,  \varphi)_K,\qquad
 \forall \varphi \in P_r(K).
 \end{equation}

The notation of discrete weak Laplacian was first introduced in
\cite{mwy0927} in conjunction with the study of plate bending
problems.

\section{Primal-Dual WG Algorithm}\label{Section:WGFEM}
Let ${\cal T}_h$ be a finite element partition of the domain
$\Omega$ into polygons in 2D or polyhedra in 3D. Assume that ${\cal
T}_h$ is shape regular in the sense described as in \cite{wy3655}.
Denote by ${\mathcal E}_h$ the set of all edges or flat faces in
${\cal T}_h$ and ${\mathcal E}_h^0={\mathcal E}_h \setminus
\partial\Omega$ the set of all interior edges or flat faces. Denote
by $h_T$ the diameter of $T\in {\cal T}_h$ and $h=\max_{T\in {\cal
T}_h}h_T$ the meshsize of the finite element partition ${\cal
T}_h$.

For any given integer $k\geq 1$ and $T\in\T_h$, define a local
weak finite element space as follows:
$$
V(k,T)=\{\{v_0,v_b,v_n\bn\}:\ v_0\in P_k(T),v_b\in P_k(e), v_n\in
P_{k-1}(e), e\subset \partial T\}.
$$
By patching $V(k,T)$ over all the elements $T\in {\cal T}_h$ through
a common value $v_b$ and $v_n\bn$ on the interior interface $\E_h^0$,
we obtain a global weak finite element space:
$$
V_h=\big\{\{v_0,v_b,v_n\bn\}:\ \{v_0,v_b,v_n\bn\}|_T\in V(k,T),
\forall T\in {\cal T}_h \big\}.
$$

For any interior edge/face $e\in\E_h^0$, by definition, there
exist two elements $T_1$ and $T_2$ sharing $e$ as a common
edge/face. Thus, any finite element function $v\in V_h$ would
satisfy the following property
$$
(v_n \bn)|_{\partial T_1 \cap e} = (v_n \bn)|_{\partial T_2 \cap e},
$$
where the left-hand side (respectively, right-hand side) stands for the value of $v_n\bn$ as seen from the element $T_1$ (respectively, $T_2$). As the two normal directions are opposite to
each other, it follows that
$$
(v_n)|_{\partial T_1 \cap e} + (v_n)|_{\partial T_2 \cap e}= 0.
$$

Next, we introduce an auxiliary finite element space as follows:
$$
W_h=\{w: \ w|_T\in P_{k-2}(T),  T\in {\cal T}_h\}.
$$
Denote by ${\cal Q}_h$ the $L^2$ projection operator onto the finite element
space $W_h$. For any $v\in V_h$, the discrete weak Laplacian, denoted by $\Delta_{w,h}v$, is computed by applying the discrete weak Laplacian $\Delta_{w,k-2,T}$ to $v$ locally on each element; i.e.,
$$
(\Delta_{w,h}v)|_T = \Delta _{w,k-2,T} (v|_T).
$$

For each edge/face $e\subset\partial T$, denote by $Q_b$ and $Q_n$ the
$L^2$ projection operators onto $P_{k}(e)$ and $P_{k-1}(e)$, respectively.
Let $V_h^g$ be the hyperplane of $V_h$ consisting of all the
finite element functions with the following boundary values on
$\Gamma_d$ and $\Gamma_n$; i.e.,
$$
V_h^g=\big\{\{v_0,v_b,v_n\bn\}\in V_h: \ v_b|_{\Gamma_d}=Q_b g_1, v_n|_{\Gamma_n}=Q_n
g_2\big\}.
$$
When $g_1=0$ and $g_2=0$, the corresponding hyperplane becomes to be
a closed subspace of $V_h$, and we denote this subspace by $V_h^0$.

\medskip

Introduce a bilinear form on $V_h\times V_h$ as follows
\begin{equation*}
s(\sigma, v)=\sum_{T\in {\cal T}_h}s_T(\sigma,v),\qquad \sigma, v\in V_h,
\end{equation*}
where
\begin{equation*}
s_T(\sigma, v)=h_T^{-3}\int_{\partial T}
(\sigma_0-\sigma_b)(v_0-v_b)ds+ h_T^{-1}\int_{\partial T} (\nabla
\sigma_0 \cdot \bn-\sigma_n ) (\nabla v_0\cdot\bn -v_n)ds.
\end{equation*}

The elliptic Cauchy problem (\ref{EQ:model-problem}) can be
discretized as a constrained minimization problem as follows: {\em
Find $u_h\in V_h^g$ satisfying
\begin{equation}\label{opti}
u_h=\argmin_{v\in V_h^g, \Delta_{w,h} v= {\cal Q}_h f} \bigg(\frac{1}{2}
s(v,v) \bigg).
\end{equation}
}

As ${\cal Q}_h$ is the $L^2$ projection onto $W_h$, the operator
equation $\Delta_{w,h} v= {\cal Q}_h f$ can be rewritten as
$$
(\Delta_{w,h} v, w) = (f, w),\qquad \forall w\in W_h.
$$
By using a Lagrange multiplier $\lambda_h\in W_h$, the constrained minimization problem (\ref{opti}) can be reformulated in
the Euler-Lagrange form: {\em Find $u_h\in V_h^g$ and
$\lambda_h\in W_h$ such that
\begin{eqnarray}\label{PD-WG32-1}
s(u_h, v)+(\Delta_{w,h} v, \lambda_h)&=&0,\qquad \quad \forall v \in V^0_h,\\
(\Delta_{w,h} u_h, w)&=&(f,w), \quad \forall w\in W_h.\label{PD-WG32-2}
\end{eqnarray}
}

The equations (\ref{PD-WG32-1}) and (\ref{PD-WG32-2}) constitute the
primal-dual weak Galerkin finite element scheme for the
elliptic Cauchy  problem (\ref{EQ:model-problem}). The equation
(\ref{PD-WG32-2}) is for the primal variable $u_h$, while
(\ref{PD-WG32-1}) is a stabilized version for the dual variable $\lambda_h$. The primal and the dual equations are integrated together by the stabilizer $s(\cdot,\cdot)$. Primal-dual finite element methods have been successfully developed for the second order elliptic equation in nondivergence form in \cite{ww2016} and Fokker-Planck type equations in \cite{ww2017}. The very same approach can also be seen in Bauman \cite{Burman} for nonsymmetric, noncoercive and ill-posed problems in a different context where the method was named the {\em stabilized finite element methods}.

The following result is well-known, see for example \cite{Gilbarg-Trudinger}.

\begin{lemma}\label{unilem}  Assume that $\Omega$ is an open bounded and connected domain in $\mathbb R^d \ (d=2,3)$ with Lipschitz continuous boundary $\Gamma=\partial\Omega$. Denote by $\Gamma_d$ the portion of the Dirichlet boundary and $\Gamma_n$ the Neumann portion. Assume that $\Gamma_d\cap \Gamma_n$ is a non-trivial portion of $\Gamma$. Then, the solutions of the following elliptic Cauchy problem, if they exist, are unique
\begin{equation*}
\begin{split}
\Delta u=&f, \qquad\mbox{\ in}\quad \Omega,\\
  u=&g_1,\qquad \mbox{on}\quad   \Gamma_d,\\
 \partial_\bn u =&g_2, \qquad\mbox{on}\quad  \Gamma_n.
\end{split}
\end{equation*}
\end{lemma}

\begin{theorem}\label{thmunique1} Assume that $\Gamma_d\cap\Gamma_n$ contains a nontrivial portion of the domain boundary $\Gamma=\partial\Omega$ and $\Gamma_d\cap \Gamma_n\Subset \Gamma$ is a proper closed subset. Then the primal-dual weak Galerkin finite element algorithm
(\ref{PD-WG32-1})-(\ref{PD-WG32-2}) has one and only one solution pair
$(u_h; \lambda_h) \in V_h^g \times W_h$.
%
\end{theorem}

\begin{proof} Since the number of equations is the same as the
number of unknowns in the system of linear equations
(\ref{PD-WG32-1})-(\ref{PD-WG32-2}), then solution existence is
equivalent to the uniqueness. To verify the uniqueness, we consider the elliptic Cauchy problem with homogeneous data (i.e., $f\equiv 0$, $g_1\equiv 0$, and $g_2\equiv 0$). If $(u_h;\lambda_h)$ is the corresponding numerical solution, then $u_h\in V_h^0$ and we may choose $v=u_h$ and $w=\lambda_h$ in (\ref{PD-WG32-1}) and (\ref{PD-WG32-2}) to obtain
$$
s(u_h, u_h)=0,
$$
which leads to $u_0=u_b$ and $\nabla u_0\cdot \bn=u_n$ on all the edges $e\in \E_h$. It follows that $u_0\in C^1(\Omega)$ is a $C^1$-conforming element. Thus, by using (\ref{2.4new2}) we obtain
\begin{equation}\label{EQ:Nov26:001}
(\Delta u_0, w)_T = (\Delta_{w,h} u_h, w)_T=0,\qquad \forall w\in
P_{k-2}(T),
\end{equation}
where we have used the equation (\ref{PD-WG32-2}). As $\Delta
u_0|_T\in P_{k-2}(T)$ on each element $T$ and $u_0 \in C^1(\Omega)$, we
then have from (\ref{EQ:Nov26:001})
$$
\Delta u_0 = 0, \qquad \mbox{in} \ \Omega,
$$
which, together with the fact that $u_0=0$ on $\Gamma_d$ and $\nabla
u_0\cdot\bn=0$ on $\Gamma_n$, yields $u_0 \equiv 0$ in $\Omega$ by Lemma \ref{unilem}.

It remains to show that $\lambda_h\equiv 0$ in $\Omega$. To this end, from the
equation (\ref{PD-WG32-1}) and the fact that $u_h\equiv 0$ in $\Omega$ we have
$$
(\Delta_{w,h} v, \lambda_h) = 0,\qquad \forall v\in V_h^0.
$$
It follows from (\ref{2.4}) that
\begin{eqnarray*}
0&=&(\Delta_{w,h} v, \lambda_h)\\
&=&\sum_{T\in\T_h}(\Delta_{w,k-2,T} v, \lambda_h)_T\\
&=&\sum_{T\in\T_h} (v_0, \Delta \lambda_h)_T -\langle
v_b,\nabla \lambda_h\cdot \bn \rangle_{\partial T} +
 \langle v_n,\lambda_h\rangle_{\partial T} \\
 &=&\sum_{T\in\T_h}(v_0, \Delta \lambda_h)_T - \sum_{e\in\E_h/\Gamma_d}
 \langle v_b, \ljump{\nabla\lambda_h}\rjump\cdot\bn_e\rangle_e +
 \sum_{e\in\E_h/\Gamma_n} \langle v_n, \ljump{\lambda_h}\rjump\rangle_e
\end{eqnarray*}
for all $v\in V_h^0$, where we have used the fact that $v_b=0$ on $\Gamma_d$ and $v_n=0$ on $\Gamma_n$. Here, $\ljump{ \lambda_h}\rjump$ is the jump across the edge/face $e\in {\cal E}_h$; more precisely, it is defined as $\ljump{\lambda_h}\rjump=\lambda_h|_{T_1}-\lambda_h|_{T_2}$ whereas $e$ is the shared edge/face of the elements $T_1$ and $T_2$ and $\ljump{\lambda_h}\rjump=\lambda_h$ whereas $e\subset\partial\Omega$. The order of $T_1$ and $T_2$ is non-essential as long as the difference
is taken in a consistent way in all the formulas. By letting $v_0= \Delta \lambda_h$ on each
element $T$ and $v_b= -\ljump{\nabla\lambda_h}\rjump\cdot\bn_e$ on each edge/face $e\in \E_h/\Gamma_d$ and
$v_n=\ljump{\lambda_h}\rjump$ on each $e\in \E_h/\Gamma_n$ in the above equation, we obtain
\begin{eqnarray}\label{EQ:Nov26:002}
\Delta \lambda_h &=& 0,\qquad \mbox{on each } T\in \T_h,\\
\ljump{\nabla\lambda_h}\rjump\cdot\bn_e & = & 0,\qquad \mbox{on each edge/face} \ e\in \E_h/\Gamma_d,
\label{EQ:Nov26:003}\\
\ljump{\lambda_h}\rjump & = & 0,\qquad \mbox{on each edge/face} \
e\in \E_h/\Gamma_n.\label{EQ:Nov26:004}
\end{eqnarray}
The equations (\ref{EQ:Nov26:003}) and (\ref{EQ:Nov26:004}) indicate
that $\lambda_h\in C^1(\Omega)$ and
$\nabla\lambda_h\cdot\bn=0$ on $\Gamma_d^c$ and $\lambda_h=0$ on $\Gamma_n^c$, where $\Gamma_d^c= \E_h/\Gamma_d$ and $\Gamma_n^c= \E_h/\Gamma_n$.
Thus, the equation (\ref{EQ:Nov26:002}) holds true in the whole
domain $\Omega$. Since $\Gamma_d \cup \Gamma_n \Subset\Gamma$ is a closed proper subset, then $\Gamma_d^c \cap \Gamma_n^c= (\Gamma_d \cup \Gamma_n)^c$ contains a nontrivial portion of $\Gamma$. Thus, from Lemma \ref{unilem}, we have $\lambda_h\equiv 0$ in $\Omega$.
This completes the proof of the theorem.
\end{proof}

\section{Stability Conditions}\label{Section:stability}
For any $v\in V_h$, let
\begin{equation}\label{norm}
\| v \|_{2,h}:=\big( \sum_{T\in {\cal T}_h}\|\Delta v_0\|^2_T
+s_T(v,v)\big)^{\frac{1}{2}}.
\end{equation}
It is easy to see that $\|\cdot\|_{2,h}$ defines a semi-norm in the
weak finite element space $V_h$. The following lemma shows that $\|
\cdot \|_{2,h}$ is indeed a norm in the subspace $V_h^0$.

\begin{lemma} The semi-norm $\|\cdot \|_{2,h}$ given as in (\ref{norm}) defines a norm in the linear space $V_h^0$.
\end{lemma}
\begin{proof} It suffices to verify the positivity property for $\| \cdot \|_{2,h}$ in the linear space
$V_h^0$. To this end, let $v=\{v_0,v_b,v_n\}\in V_h^0$ satisfy
$\|v \|_{2,h}=0$. It follows that on each element $T$ we have
\begin{eqnarray}\label{EQ:Nov26:010}
\Delta v_0 &=& 0, \qquad \mbox{ in } T,\\
v_0-v_b & = & 0,\qquad \mbox{ on } \partial T,
\label{EQ:Nov26:011}\\
v_n - \nabla v_0\cdot\bn & = & 0,\qquad \mbox{ on } \partial T.
\label{EQ:Nov26:012}
\end{eqnarray}
Since both $v_b$ and $v_n\bn$ are single-valued on each edge/face
$e\in\E_h$, the equations (\ref{EQ:Nov26:011}) and
(\ref{EQ:Nov26:012}) imply $v_0\in C^1(\Omega)$. Hence, the
equation (\ref{EQ:Nov26:010}) holds true in the whole domain
$\Omega$. This, combined with the fact that $v_b=0$ on $\Gamma_d$ and $v_n=0$ on $\Gamma_n$, shows $v_0\equiv 0$ as $\Gamma_d\cap\Gamma_n$ contains a nontrivial portion of the domain boundary $\Gamma$. Finally, it follows from
(\ref{EQ:Nov26:011}) and (\ref{EQ:Nov26:012}) that $v_b\equiv 0$ and
$v_n\equiv 0$. This completes the proof of the lemma.
\end{proof}

The following result shows that the weak Laplacian operator is bounded with respect to the semi-norm $\|\cdot\|_{2,h}$.

\begin{lemma} The following boundedness estimate holds true for the discrete weak Laplacian $\Delta_w$:
\begin{equation}\label{EQ:November:26:800}
\|\Delta_{w,h} v\| \lesssim \|v\|_{2,h},\qquad \forall v\in V_h.
\end{equation}
\end{lemma}

\begin{proof} From (\ref{2.4new}) we have
\begin{equation*}
(\Delta_{w,h} v,\varphi)_T  = (\Delta  v_0, \varphi)_T+
 \langle v_0-v_b,\nabla \varphi\cdot \bn \rangle_{\partial T}+
 \langle v_n-\nabla v_0\cdot \bn,\varphi  \rangle_{\partial T}
 \end{equation*}
for all $\varphi\in P_{k-2}(T)$. Now using the Cauchy-Schwarz inequality, the trace inequality (\ref{tracenew}) and the inverse inequality, we obtain
\begin{eqnarray*}
|(\Delta_{w,h} v, \varphi)_T| & \leq &|(\Delta  v_0, \varphi)_T|+
 |\langle v_0-v_b,\nabla \varphi\cdot \bn \rangle_{\partial T}|+
 |\langle v_n-\nabla v_0\cdot \bn, \varphi \rangle_{\partial
 T}|\\
 &\leq & \|\Delta  v_0\|_T \|\varphi \|_T +
 \|v_0-v_b\|_{\partial T} \|\nabla\varphi\|_{\partial T} + \|v_n-\nabla v_0\cdot
 \bn\|_{\partial T} \|\varphi\|_{\partial T}\\
 &\lesssim &\|\Delta v_0\|_T \|\varphi\|_T + h_T^{-1/2}\|v_0-v_b\|_{\partial T}
 \|\nabla\varphi\|_{T}+h_T^{-1/2}\|v_n-\nabla v_0\cdot\bn\|_{\partial T}
 \|\varphi\|_T\\
 &\lesssim & \left(\|\Delta v_0\|_T + h_T^{-3/2}\|v_0-v_b\|_{\partial T}+
 h_T^{-1/2}\|v_n-\nabla v_0\cdot\bn\|_{\partial T} \right)
 \|\varphi\|_T,
 \end{eqnarray*}
 which leads to
 $$
\|\Delta_{w,h} v\|_T \lesssim \|\Delta v_0\|_T +
h_T^{-3/2}\|v_0-v_b\|_{\partial T}+
 h_T^{-1/2}\|v_n-\nabla v_0\cdot\bn\|_{\partial T}.
 $$
Summing the square of the above inequality over all the element
$T\in\T_h$ gives rise to the estimate (\ref{EQ:November:26:800}).
This completes the proof of the lemma.
\end{proof}

\begin{lemma}\label{lembou} The following boundedness estimates hold true:
\begin{align} \label{bound1}
|s(\sigma,v)| \le & \|\sigma\|_{2,h} \| v\|_{2,h}, \qquad \sigma, v\in V_h, \\
|(\Delta_{w,h} v, w)| \lesssim & \|v\|_{2,h} \|w\|,  \qquad v\in V_h,
w\in W_h. \label{bound2}
\end{align}
\end{lemma}

\begin{proof}
To derive (\ref{bound1}), we use the Cauchy-Schwarz inequality to
obtain
\begin{eqnarray*}
|s(\sigma,v)| &=& \Big |\sum_{T\in {\cal T}_h}h_T^{-3}\langle
\sigma_0-\sigma_b, v_0-v_b\rangle_\pT+ h_T^{-1}\langle \nabla
\sigma_0 \cdot \bn-\sigma_n, \nabla v_0\cdot\bn -v_n\rangle_\pT
\Big | \\
&\leq & \left(\sum_{T\in {\cal T}_h}h_T^{-3}
\|\sigma_0-\sigma_b\|_\pT^2\right)^{1/2}  \left(\sum_{T\in {\cal T}_h}h_T^{-3} \|v_0-v_b\|_\pT^2\right)^{1/2}   \\
& & +   \left(\sum_{T\in {\cal T}_h} h_T^{-1}\|\nabla \sigma_0 \cdot
\bn-\sigma_n\|_\pT^2\right)^{1/2}\left(\sum_{T\in {\cal T}_h}
h_T^{-1}\|\nabla v_0 \cdot
\bn-v_n\|_\pT^2\right)^{1/2} \\
&\leq& \| \sigma\|_{2,h} \|v\|_{2,h}.
\end{eqnarray*}
As to (\ref{bound2}), we use  the Cauchy-Schwarz inequality  and the boundedness estimate
(\ref{EQ:November:26:800}) to obtain
\begin{equation*}
|(\Delta_{w,h} v,w)| \leq   \left(\sum_{T\in {\cal T}_h} \| \Delta_{w,h} v\|_T^2 \right)^{1/2}
 \left(\sum_{T\in {\cal T}_h}\|w\|_T^2\right)^{1/2} \lesssim \|v\|_{2,h}
\|w\|.
\end{equation*}
This completes the proof of the lemma.
\end{proof}

The kernel of the weak Laplacian in $V_h$ is a subspace given by
$$
Z_h= \{v\in V_h: \ \Delta_{w,h} v = 0\}.
$$
For any $v\in Z_h$, we have $\Delta_{w,h} v=0$ so that on each element
$T\in\T_h$
$$
(\Delta_{w,h} v, \varphi)_T = 0\qquad \forall \varphi\in P_{k-2}(T).
$$
From (\ref{2.4new}) we have
\begin{equation*}
0=(\Delta_{w,h} v,\varphi)_T  = (\Delta  v_0, \varphi)_T+
 \langle v_0-v_b,\nabla \varphi\cdot \bn \rangle_{\partial T}+
 \langle v_n-\nabla v_0\cdot \bn,\varphi  \rangle_{\partial T}.
 \end{equation*}
It follows that
\begin{equation*}
(\Delta  v_0, \varphi)_T=
 \langle v_b-v_0,\nabla \varphi\cdot \bn \rangle_{\partial T}+
 \langle \nabla v_0\cdot \bn-v_n,\varphi  \rangle_{\partial T}.
 \end{equation*}
Using  the Cauchy-Schwarz inequality, the trace inequality (\ref{tracenew}) and the inverse inequality, we
arrive at
\begin{eqnarray*}
|(\Delta v_0, \varphi)_T| & \leq &
 |\langle v_b-v_0,\nabla \varphi\cdot \bn \rangle_{\partial T}|+
 |\langle \nabla v_0\cdot \bn - v_n, \varphi \rangle_{\partial
 T}|\\
 &\leq &
 \|v_0-v_b\|_{\partial T} \|\nabla\varphi\|_{\partial T} + \|v_n-\nabla v_0\cdot
 \bn\|_{\partial T} \|\varphi\|_{\partial T}\\
 &\lesssim & h_T^{-1/2}\|v_0-v_b\|_{\partial T}
 \|\nabla\varphi\|_{T}+h_T^{-1/2}\|v_n-\nabla v_0\cdot\bn\|_{\partial T}
 \|\varphi\|_T\\
 &\lesssim & \left(h_T^{-3/2}\|v_0-v_b\|_{\partial T}+
 h_T^{-1/2}\|v_n-\nabla v_0\cdot\bn\|_{\partial T} \right)
 \|\varphi\|_T.
 \end{eqnarray*}
 Thus, we have
 \begin{equation}\label{EQ:November:26:801}
\|\Delta v_0\|_T \lesssim h_T^{-3/2}\|v_0-v_b\|_{\partial T}+
 h_T^{-1/2}\|v_n-\nabla v_0\cdot\bn\|_{\partial T}
\end{equation}
for any $v\in Z_h$. Summing (\ref{EQ:November:26:801}) over all the
element $T\in\T_h$ yields
\begin{equation}\label{EQ:November:26:802}
\sum_{T\in {\cal T}_h}\|\Delta v_0\|_T^2 \lesssim s(v,v),\qquad v\in Z_h.
\end{equation}
Consequently, we have proved the following coercivity result for the
bilinear form $s(\cdot,\cdot)$ in $Z_h\times Z_h$.

\begin{lemma}\label{lemcoe} There exists a constant $\alpha>0$ such that
\begin{equation}\label{coer}
s(v,v)\geq \alpha \| v\|_{2,h}^2, \qquad \forall v\in Z_h.
\end{equation}
\end{lemma}

In the auxiliary finite element space $W_h$, we introduce the
following norm
\begin{equation}\label{EQ:NormInWh}
\|\lambda\|_{0,h} = \left(\sum_{T\in\T_h} h_T^4\|\Delta
\lambda\|_T^2 + \sum_{e\in(\E_h/\Gamma_d)} h_e^3
\|\ljump{\nabla\lambda}\rjump\|_e^2 + \sum_{e\in(\E_h/\Gamma_n)} h_e
\|\ljump{\lambda}\rjump\|_e^2\right)^{1/2},
\end{equation}

\begin{lemma} \label{leminf}(inf-sup condition) For any $\lambda \in W_h$, there exists $v^*\in V_h^0$ satisfying
\begin{align}\label{inf1}
 (\Delta_{w,h} v^*, \lambda)&  = \|\lambda\|_{0,h}^2,  \\
 \|v^*\|_{2,h}^2& \lesssim \|\lambda\|^2_{0,h}.\label{inf2}
\end{align}
\end{lemma}

\begin{proof} On each element $T\in\T_h$, from (\ref{2.4}) we have
\begin{equation*}
(\Delta_{w, k-2, T} v,\lambda)_T  = ( v_0, \Delta \lambda)_T-
 \langle v_b,\nabla \lambda\cdot \bn \rangle_{\pT}+
 \langle v_n,\lambda \rangle_{\pT}
 \end{equation*}
for any $v\in V_h^0$. Summing over all $T\in\T_h$ yields
\begin{eqnarray*}
(\Delta_{w,h} v, \lambda)&=&\sum_{T\in\T_h}(\Delta_{w,k-2,T} v, \lambda)_T\\
&=&\sum_{T\in\T_h} (v_0, \Delta \lambda)_T -\langle v_b,\nabla
\lambda\cdot \bn \rangle_{\pT} +\langle v_n,\lambda\rangle_{\pT} \\
 &=&\sum_{T\in\T_h}(v_0, \Delta \lambda)_T - \sum_{e\in(\E_h/\Gamma_d)}
 \langle v_b, \ljump{\nabla\lambda}\rjump\cdot\bn_e\rangle_e +
 \sum_{e\in(\E_h/\Gamma_n)} \langle v_n, \ljump{\lambda}\rjump\rangle_e,
\end{eqnarray*}
where we have used $v_b=0$ on $\Gamma_d$ and $v_n=0$ on $\Gamma_n$.
By setting
\begin{eqnarray*}
v^*_0 &=& h_T^4 \Delta \lambda,\qquad\qquad\ \mbox{in } \ T\in\T_h,\\
v^*_b &=& -h_e^3\ljump{\nabla\lambda}\rjump\cdot\bn_e,\quad
\mbox{on}\ e\in \E_h/\Gamma_d,\\
v^*_n &=& h_e \ljump{\lambda}\rjump,\qquad\qquad\
\mbox{on}\;\; e\in \E_h/\Gamma_n,\\
v^*_b &=& 0,\qquad\qquad\qquad
\mbox{on}\; e\in \Gamma_d,\\
v^*_n &=& 0,\qquad\qquad\qquad
\mbox{on}\; e\in \Gamma_n,\\
\end{eqnarray*}
we have a weak finite element function $v^*=\{v_0^*, v_b^*, v_n^*\bn\}\in
V_h^0$ satisfying
\begin{eqnarray*}
(\Delta_{w,h} v^*, \lambda)=\sum_{T\in\T_h} h_T^4\|\Delta \lambda\|_T^2
+ \sum_{e\in(\E_h/\Gamma_d)} h_e^3
\|\ljump{\nabla\lambda}\rjump\|_e^2 + \sum_{e\in(\E_h/\Gamma_n)} h_e
\|\ljump{\lambda}\rjump\|_e^2,
\end{eqnarray*}
which leads to (\ref{inf1}). The boundedness estimate (\ref{inf2})
can be verified by using the norm definition (\ref{norm}) and the
standard inverse inequality without any difficulty; details are left to interested readers as an exercise. This completes the proof of the lemma.
\end{proof}

\section{Convergence Analysis}\label{Section:convergence}
The goal of this section is to establish a convergence theory for the
solution of the primal-dual weak Galerkin algorithm
(\ref{PD-WG32-1})-(\ref{PD-WG32-2}) under the assumption that the
continuous problem (\ref{EQ:model-problem}) has a solution $u$
that is sufficiently regular for us to perform all the necessary
mathematical operations and estimates.

\subsection{Error equations}
On each element $T\in\T_h$, denote by $Q_0$ the $L^2$ projection
operator onto $P_k(T)$. For any $\theta\in H^2(\Omega)$, denote by $Q_h
\theta$ the $L^2$ projection onto the weak finite element space
$V_h$ such that on each element $T$,
$$
Q_h\theta=\{Q_0\theta,Q_b\theta, Q_n(\nabla \theta\cdot \bn)\bn\}.
$$
As was shown in \cite{mwy0927}, the following commutative property
holds true:
\begin{equation}\label{EQ:CommutativeProperty}
\Delta _{w,h}(Q_h \theta) = {\cal Q}_h(\Delta \theta),\qquad \theta
\in H^2(T).
\end{equation}

Now let $(u_h;\lambda_h)\in V^g_h \times W_h$
be the numerical solution arising from the primal-dual weak Galerkin
algorithm (\ref{PD-WG32-1})-(\ref{PD-WG32-2}). Denote the error
functions by
\begin{align}\label{error}
e_h&=u_h-Q_h u,\\
\epsilon_h&=\lambda_h - 0.\label{error-2}
\end{align}
Note that the Lagrange multiplier is trivial ($\lambda=0$) for the continuous problem.

\begin{lemma}\label{errorequa}
Let $u$ be the solution of the elliptic Cauchy problem
(\ref{EQ:model-problem}) and $(u_h;\lambda_h)\in V^g_h \times W_h$ be
its numerical approximation arising from the primal-dual weak
Galerkin algorithm (\ref{PD-WG32-1})-(\ref{PD-WG32-2}). Then, the error functions $e_h$ and $\epsilon_h$ defined in (\ref{error})-(\ref{error-2}) satisfy the following equations
\begin{eqnarray}\label{sehv}
 s( e_h, v)+(\Delta_{w,h} v, \epsilon_h) &=&-s(Q_h u, v),\qquad \forall v\in V_h^{0},\\
 (\Delta_{w,h} e_h, w)&=&0,\qquad\qquad\qquad \forall w\in W_h. \label{sehv2}
\end{eqnarray}
\end{lemma}

\begin{proof}
First, by subtracting $s(Q_hu, v)$ from both sides of (\ref{PD-WG32-1}) we have
\begin{equation*}
\begin{split}
s( u_h-Q_h u, v)+(\Delta_{w,h} v, \lambda_h) =-s(Q_h u, v), \qquad
\forall v\in V_h^0,
\end{split}
\end{equation*}
which leads to the first error equation (\ref{sehv}).
As to (\ref{sehv2}), we use the equation (\ref{PD-WG32-2}) and the
commutative property (\ref{EQ:CommutativeProperty}) to obtain
\begin{eqnarray*}
 (\Delta_{w,h} e_h, w) & = & (\Delta_{w,h}(u_h-Q_hu), w)\\
 &= & (\Delta_{w,h} u_h, w) - (\Delta_{w,h} Q_h u, w) \\
 & = & (f, w) - ( {\cal{Q}}_h \Delta u, w) \\
 & = & (f,w) - ( {\cal{Q}}_h f, w)\\
& = & 0
\end{eqnarray*}
for all $w\in W_h$. This completes the proof of the lemma.
\end{proof}

\subsection{Error estimates}
Assume that the finite element partition ${\cal T}_h$ is
shape-regular. Thus, on each $T\in {\cal T}_h$ the
following trace inequality holds true \cite{wy3655}:
\begin{equation}\label{tracein}
 \|\phi\|^2_{\partial T} \lesssim   h_T^{-1}\|\phi\|_T^2+h_T \|\nabla
 \phi\|_T^2,\quad \phi\in H^1(T).
\end{equation}
If $\phi$ is additionally a polynomial function on the element $T\in
{\cal T}_h$, we have from (\ref{tracein}) and the inverse inequality
(see \cite{wy3655} for details on arbitrary polygonal elements) that
\begin{equation}\label{tracenew}
 \|\phi\|^2_{\partial T} \lesssim  h_T^{-1}\|\phi\|_T^2.
\end{equation}

The following results can be found in \cite{mwy3655}.

\begin{lemma} Let ${\cal T}_h$ be shape regular. Then, for any $0\leq s \leq 2$ and
$1\leq m \leq k$, one has
\begin{equation}\label{error1}
 \sum_{T\in {\cal T}_h}h_T^{2s}\|u-Q_0u\|^2_{s,T}\lesssim   h^{2(m+1)}\|u\|^2_{m+1},
\end{equation}
\begin{equation}\label{error3}
 \sum_{T\in {\cal T}_h}h_T^{2s}\|u-{\cal Q}_hu\|^2_{s,T}\lesssim   h^{2(m-1)}\|u\|^2_{m-1}.
\end{equation}
 \end{lemma}

\begin{theorem} \label{theoestimate}
Let $u$ be the exact solution of the elliptic Cauchy problem
(\ref{EQ:model-problem}), and $(u_h;\lambda_h)\in V_h^g \times W_h$
be its numerical approximation arising from the primal-dual weak
Galerkin algorithm (\ref{PD-WG32-1})-(\ref{PD-WG32-2}) of order
$k\ge 2$. Assume that the exact solution is so regular that $u\in
H^{k+1}(\Omega)$. Then, the following error estimate holds true:
 \begin{equation}\label{erres}
\|u_h - Q_h u \|_{2,h} + \| \lambda_h\|_{0,h} \lesssim h^{k-1}
\|u\|_{k+1}.
\end{equation}
\end{theorem}

\begin{proof} We use the error equations (\ref{sehv})-(\ref{sehv2})
to derive the error estimate (\ref{erres}). To this end, note that
the right-hand side in the equation (\ref{sehv}) is given by
\begin{equation}\label{sterm}
\begin{split}
s(Q_hu,v)=&\sum_{T\in {\cal T}_h}h_T^{-3} \langle
 Q_0u-Q_bu, v_0-v_b\rangle _{\partial T} \\
 &+ h_T^{-1} \langle  \nabla Q_0u \cdot \bn-Q_n(\nabla u \cdot \bn), \nabla Q_0 v\cdot\bn -Q_n (\nabla v \cdot \bn)\rangle _{\partial T}.
\end{split}
\end{equation}
From the Cauchy-Schwarz inequality, the trace
inequality (\ref{tracein}) and the estimate (\ref{error1}) we have
\begin{equation}\label{sterm1}
\begin{split}
& \Big| \sum_{T\in {\cal T}_h}h_T^{-3} \langle
 Q_0u-Q_bu, v_0-v_b\rangle _{\partial T} \Big|\\
 \leq & \Big| \sum_{T\in {\cal T}_h}h_T^{-3} \langle
 Q_0u- u, v_0-v_b\rangle _{\partial T} \Big|\\
 \leq & \Big(  \sum_{T\in {\cal T}_h}h_T^{-3}\|
 Q_0u- u\|^2_{\partial T}\Big)^{\frac{1}{2}}  \Big(  \sum_{T\in {\cal T}_h}h_T^{-3}\|  v_0-v_b\|^2_{\partial T}\Big)^{\frac{1}{2}}  \\
\lesssim & \Big(  \sum_{T\in {\cal T}_h}h_T^{-4}\|
 Q_0u- u\|^2_{T}+h_T^{-2}\|
 Q_0u- u\|^2_{1, T}\Big)^{\frac{1}{2}} s(v,v)^{\frac12} \\
 \lesssim & h^{k-1} \|u\|_{k+1} s(v,v)^{\frac12}.
 \end{split}
\end{equation}
Analogously, the second term on the right-hand side of (\ref{sterm})
can be bounded as follows:
\begin{equation}\label{sterm2}
\begin{split}
&\Big| \sum_{T\in {\cal T}_h}h_T^{-1} \langle  \nabla Q_0u \cdot
\bn-Q_n(\nabla u \cdot \bn), \nabla Q_0 v\cdot\bn -Q_n (\nabla v
\cdot \bn)\rangle _{\partial T}\Big|\\
&\lesssim  h^{k-1} \|u\|_{k+1} s(v,v)^{\frac12}.
\end{split}
\end{equation}
Substituting (\ref{sterm1}) and (\ref{sterm2}) into  (\ref{sterm})
gives
\begin{equation}\label{EQ:Estimate_4_RHT}
\Big| s(Q_hu,v)\Big|  \lesssim h^{k-1} \|u\|_{k+1} s(v,v)^{\frac12}.
\end{equation}

Going back to the error estimate, we observe that the error function
$e_h=u_h-Q_h u$ belongs to $V_h^0$. Thus, we may take $v=e_h$ in the
error equation (\ref{sehv}) to obtain
$$
s(e_h, e_h) + (\Delta_{w,h} e_h, \lambda_h) = - s(Q_h u, e_h).
$$
Note that the second equation (\ref{sehv2}) implies $\Delta_{w,h} e_h
\equiv 0$; i.e., $e_h\in Z_h$. Thus,
$$
s(e_h, e_h) = - s(Q_h u, e_h).
$$
Now from the estimate (\ref{EQ:Estimate_4_RHT}) we obtain
$$
s(e_h, e_h) = |s(Q_h u, e_h)| \lesssim h^{k-1} \|u\|_{k+1}
 s(e_h,e_h)^{\frac12},
$$
which leads to
\begin{equation}\label{EQ:error-estimate-1}
s(e_h,e_h)^{\frac12} \lesssim h^{k-1} \|u\|_{k+1}.
\end{equation}
Using the coercivity estimate (\ref{coer}) we arrive at
\begin{equation}\label{EQ:error-estimate-2}
\|e_h\|_{2,h}^2 \lesssim s(e_h,e_h) \lesssim h^{2k-2} \|u\|_{k+1}^2.
\end{equation}

It remains to estimate the Lagrange multiplier or the dual variable $\lambda_h$.
From the first error equation (\ref{sehv}) we have
\begin{equation}\label{EQ:November:26:900}
(\Delta_{w,h} v, \lambda_h) = - s(e_h, v) - s (Q_h u, v),\qquad \forall
v\in V_h^0.
\end{equation}
According to Lemma \ref{leminf}, for the given $\lambda_h$, there
exists a finite element function $v^*\in V_h^0$ satisfying
(\ref{inf1})-(\ref{inf2}). Combining this with the equation
(\ref{EQ:November:26:900}) yields
$$
\|\lambda_h\|_{0,h}^2 = (\Delta_{w,h} v^*, \lambda_h) = - s(e_h, v^*) -
s(Q_h u, v^*).
$$
It follows that
\begin{equation}\label{EQ:November:26:901}
\begin{split}
\|\lambda_h\|_{0,h}^2 & \leq |s(e_h, v^*)| + |s(Q_h u, v^*)| \\
& \leq s(e_h,e_h)^{1/2} s(v^*, v^*)^{1/2} + |s(Q_h u, v^*)| \qquad
\mbox{by Cauchy-Schwarz}\\
& \lesssim h^{k-1} \|u\|_{k+1} s(v^*, v^*)^{1/2}\qquad\qquad
\mbox{by (\ref{EQ:error-estimate-1})
and (\ref{EQ:Estimate_4_RHT})} \\
& \lesssim h^{k-1} \|u\|_{k+1} \|\lambda_h\|_{0,h}\qquad\qquad
\mbox{by (\ref{inf2})}.
\end{split}
\end{equation}
Hence, we obtain
$$
\|\lambda_h\|_{0,h} \lesssim h^{k-1} \|u\|_{k+1},
$$
which, together with the error estimate (\ref{EQ:error-estimate-2}),
yields the desired optimal order error estimate (\ref{erres}). This
completes the proof of the theorem.
\end{proof}

\section{Error Estimate in a Weak $L^2$ Topology}\label{Section:L2Error}

To establish an error estimate for
(\ref{PD-WG32-1})-(\ref{PD-WG32-2}) in $L^2$-related topology, we consider the dual problem of seeking $\phi$ satisfying
\begin{align}\label{dual1}
\Delta \phi=& \ \eta,\qquad \text{in}\quad  \Omega,\\
\phi=& \ 0,\qquad \text{on}\quad \Gamma_n^c, \label{dual2}\\
\nabla\phi\cdot \bn=& \ 0,\qquad \text{on}\quad
\Gamma_d^c,\label{dual3}
\end{align}
where $\eta\in L^2(\Omega)$. Denote by $X_\gamma$ the set of all functions $\eta\in L^2(\Omega)$ so that the problem (\ref{dual1})-(\ref{dual3}) has a solution and, furthermore, the solution has the $H^{1+\gamma}$-regularity
\begin{equation}\label{regul}
\|\phi\|_{1+\gamma}\leq C\|\eta\|_0
\end{equation}
with $1/2<\gamma\le 1$.

\begin{lemma}\label{Lemma:TechnicalEquality}
Let $\eta\in X_\gamma$. For any $v=\{v_0, v_b, v_n\}\in V_{h}^0$, we have the following identity
\begin{equation}\label{2.14:800}
\begin{split}
(\eta, v_0) = &\sum_{T\in {\cal T}_h} (\Delta_{w,h} v, \phi)_T-
 \langle v_0- v_b,\nabla  {\cal Q}_h\phi \cdot \bn  - \nabla  \phi \cdot \bn  \rangle_{\partial T}\\
&+ \langle  \nabla v_0\cdot \bn - v_n, {\cal Q}_h \phi -\phi
\rangle_{\partial T}.
\end{split}
\end{equation}
\end{lemma}

\begin{proof} By testing (\ref{dual1}) with $v_0$ on each element $T\in\T_h$,
we obtain from the usual integration by parts that
\begin{equation}\label{2.11}
\begin{split}
(\eta, v_0)= &\sum_{T\in {\cal T}_h}(\Delta \phi, v_0)_T \\
=& \sum_{T\in {\cal T}_h} (\phi, \Delta v_0)_T -\langle \phi, \nabla  v_0\cdot \bn \rangle_{\partial T}
+ \langle \nabla  \phi \cdot \bn,  v_0\rangle_{\partial T}\\
=& \sum_{T\in {\cal T}_h} (\phi, \Delta v_0)_T
-\langle \phi, \nabla v_0\cdot \bn-v_n \rangle_{\partial T}+
\langle \nabla \phi \cdot \bn, v_0 -v_b\rangle_{\partial T}\\
\end{split}
\end{equation}
where we have used the homogeneous boundary condition
(\ref{dual2})-(\ref{dual3}) and the fact that $v_b=0$ on $\Gamma_d$ and
$v_n=0$ on $\Gamma_n$ in the third line. Next, by setting
$\varphi={\cal Q}_h \phi$ in (\ref{2.4new}), we arrive at
\begin{align*}
(\Delta_{w,h} v, {\cal Q}_h \phi )_T = & (\Delta v_0,   {\cal Q}_h
\phi)_T+ \langle v_0- v_b,\nabla  {\cal Q}_h \phi\cdot \bn
\rangle_{\partial T}-\langle  \nabla v_0\cdot \bn - v_n, {\cal Q}_h \phi \rangle_{\partial T} \\
=&(\Delta v_0, \phi)_T+ \langle v_0- v_b,\nabla  {\cal Q}_h
\phi\cdot \bn \rangle_{\partial T}-
 \langle  \nabla v_0\cdot \bn - v_n, {\cal Q}_h \phi \rangle_{\partial T},
\end{align*}
which can be rewritten as
\begin{equation*}\label{2.13}
\begin{split}
(\Delta  v_0,  \phi)_T =(\Delta_{w,h} v, {\cal Q}_h \phi )_T-
 \langle v_0- v_b,\nabla  {\cal Q}_h \phi\cdot \bn \rangle_{\partial T}+
 \langle  \nabla v_0\cdot \bn - v_n, {\cal Q}_h \phi \rangle_{\partial
 T}.
\end{split}
\end{equation*}
Substituting the above identity into (\ref{2.11}) yields
\begin{equation}\label{2.14}
\begin{split}
 (\eta, v_0)=&\sum_{T\in {\cal T}_h} (\Delta_{w,h} v, {\cal Q}_h \phi)_T-
 \langle v_0- v_b,\nabla  {\cal Q}_h \phi\cdot \bn \rangle_{\partial T}\\
 &+\langle  \nabla v_0\cdot \bn - v_n, {\cal Q}_h \phi \rangle_{\partial T}-\langle \phi,
 \nabla v_0\cdot \bn-v_n \rangle_{\partial T}\\
 &+ \langle \nabla \phi \cdot \bn,   v_0 -v_b\rangle_{\partial T}\\
=&\sum_{T\in {\cal T}_h} (\Delta_{w,h} v, \phi )_T-
 \langle v_0- v_b,\nabla  {\cal Q}_h \phi \cdot \bn  - \nabla \phi \cdot \bn  \rangle_{\partial T}\\&+
 \langle  \nabla v_0\cdot \bn - v_n, {\cal Q}_h \phi -\phi  \rangle_{\partial T},\\
\end{split}
\end{equation}
which completes the proof of the lemma.
\end{proof}

We are now in a position to present an error estimate for the
component $u_0$ of the weak finite element solution $u_h$ in the weak
topology induced by the space $X_\gamma$. The result can be stated as follows.

\begin{theorem}\label{Thm:L2errorestimate}
Let $k\geq 2$ be the order of the finite element space in the numerical
scheme (\ref{PD-WG32-1})-(\ref{PD-WG32-2}) and set $\tau_0=\min\{2,k-1\}$.
Denote by $u_h\in V_h^g$ the numerical solution of (\ref{EQ:model-problem}) arising from the primal-dual WG algorithm (\ref{PD-WG32-1})-(\ref{PD-WG32-2}), with $\lambda_h\in W_h$ being the numerical Lagrange multiplier. Assume that the exact solution
$u$ exists and is sufficiently regular such that $u\in
H^{k+1}(\Omega)$. Under the $H^{1+\gamma}$-regularity assumption (\ref{regul}), the following error estimate holds true
\begin{equation}\label{e0}
\sup_{\eta\in X_\gamma}\frac{|(Q_0u - u_0, \eta)|}{\|\eta\|}  \lesssim h^{k+\tau_0+\gamma-2} \|u\|_{k+1}.
\end{equation}
\end{theorem}

\begin{proof} By letting $v=e_h$ in Lemma \ref{Lemma:TechnicalEquality}, we
have from (\ref{2.14:800})
\begin{equation}\label{2.14:800:10:L2}
\begin{split}
(\eta, e_0) =& \sum_{T\in {\cal T}_h} (\Delta_{w,h} e_h, \phi)_T -
 \langle e_0- e_b,\nabla  {\cal Q}_h \phi \cdot \bn  - \nabla \phi \cdot \bn  \rangle_{\partial T} \\
 &+ \langle  \nabla e_0\cdot \bn - e_n, {\cal Q}_h \phi - \phi  \rangle_{\partial T}.\\
\end{split}
\end{equation}
The right-hand side of (\ref{2.14:800:10:L2}) can be estimated as
follows. First, we use the error equation (\ref{sehv2}) to obtain
\begin{equation}\label{2.14.120:L2}
\begin{split}
    \sum_{T\in{\cal T}_h}
(\Delta_{w,h} e_h, \phi)_T =\sum_{T\in{\cal T}_h} (\Delta_{w,h} e_h,  {\cal
Q}_h \phi )_T = 0.
\end{split}
\end{equation}
Secondly, we use the Cauchy-Schwarz inequality, the trace inequality
(\ref{tracein}), the interpolation error estimate (\ref{error3}),
the regularity assumption (\ref{regul}) and the error estimate
(\ref{erres}) to obtain
\begin{equation}\label{EQ:New:2015:800:L2}
\begin{split}
& \left| \sum_{T\in{\cal T}_h} \langle e_0- e_b,\nabla  {\cal Q}_h \phi \cdot \bn  - \nabla \phi \cdot \bn
\rangle_{\partial T}\right| \\
\leq & \Big(\sum_{T\in{\cal T}_h} h_T^{-3} \|e_0- e_b\|^2_{\partial T} \Big)^{\frac{1}{2}}
\Big(\sum_{T\in{\cal T}_h} h_T^{3} \|\nabla  {\cal Q}_h \phi \cdot \bn  - \nabla  \phi \cdot \bn\|^2_{\partial T} \Big)^{\frac{1}{2}} \\
\lesssim &  h^{k-1} \|u\|_{k+1} h^{\tau_0-1+\gamma}\|\phi\|_{1+\gamma}\\
\lesssim & h^{k+\tau_0+\gamma-2} \|u\|_{k+1}\|\eta\|_0.
\end{split}
\end{equation}
As to the third term, we once again use the Cauchy-Schwarz
inequality, trace inequality (\ref{tracein}), the interpolation
error estimate (\ref{error3}), the regularity assumption
(\ref{regul}) and the error estimate (\ref{erres}) to obtain
\begin{equation}\label{EQ:New:2015:810:L2}
\begin{split}
&\ \left| \sum_{T\in{\cal T}_h} \langle  \nabla e_0\cdot \bn - e_n, {\cal Q}_h \phi -\phi  \rangle_{\partial T}\right|\\
\leq &\ \Big(\sum_{T\in{\cal T}_h} h_T^{-1} \|\nabla e_0\cdot \bn - e_n\|^2_{\partial T} \Big)^{\frac{1}{2}}
\Big(\sum_{T\in{\cal T}_h} h_T  \|{\cal Q}_h \phi -\phi \|^2_{\partial T} \Big)^{\frac{1}{2}} \\
\lesssim & \ \|e_h\|_{2,h} \Big(\sum_{T\in{\cal T}_h} \|{\cal Q}_h \phi -\phi \|^2_{T} +h_T^2  \|{\cal Q}_h \phi -\phi \|^2_{1, T} \Big)^{\frac{1}{2}} \\
\lesssim & \ h^{k-1} \|u\|_{k+1} h^{\tau_0+\gamma-1}\|\phi\|_{1+\gamma}\\
\lesssim & \ h^{k+\tau_0+\gamma-2} \|u\|_{k+1}\|\eta\|_0.
\end{split}
\end{equation}
Finally, by inserting the estimates (\ref{2.14.120:L2})
-(\ref{EQ:New:2015:810:L2}) into (\ref{2.14:800:10:L2}) we arrive at
$$
|(e_0, \eta)| \lesssim h^{k+\tau_0+\gamma-2} \|u\|_{k+1}\|\eta\|_0.
$$
This completes the proof of the theorem.
\end{proof}

\section{Numerical Experiments}\label{Section:NE}
In this section we shall present some
numerical results for the numerical approximations of the elliptic Cauchy problem (\ref{EQ:model-problem}) arising from the primal-dual
weak Galerkin scheme (\ref{PD-WG32-1})-(\ref{PD-WG32-2})
corresponding to the lowest order; i.e., $k=2$. For simplicity, the
domain is chosen as an unit square $\Omega=(0,1)^2$, and uniform
triangulations of $\Omega$ are employed in the numerical implementation. The finite element functions are of $C^0$-type so that $v_0=v_b$ on the boundary of each element; note that the convergence theory presented in previous sections is applicable to such elements.

The local finite element space for the primal variable is thus given by
$$
V(2,T)=\{v=\{v_0,v_n\bn\}:\ v_0\in P_2(T), v_n|_e\in P_1(e), \forall \;
  \mbox{edge}\ e\subset \pT\},
$$
and the finite element space for the Lagrange multiplier (also known as the dual variable) $\lambda$
consists of piecewise constants. For any given $v=\{v_0,v_n\bn\}\in
V(2,T)$, the action of the discrete weak Laplacian on $v$ (i.e.,
$\Delta_{w,h} v$) is computed as a constant on $T$ by using the
following equation
\begin{equation*}
(\Delta _{w,h} v, \varphi)_T= - (\nabla v_0,\nabla \varphi)_T +
 \langle v_n,\varphi \rangle_{\partial T},\qquad \forall \ \varphi \in P_0(T).
 \end{equation*}
As the test function is constant-valued on the element $T$, the above equation can be simplified as
  \begin{equation*}
  \begin{split}
 (\Delta _{w,h} v,\varphi)_T= \langle v_n,\varphi \rangle_{\pT},\qquad \forall \ \varphi \in P_0(T).
\end{split}
\end{equation*}

The error for the solution of the primal-dual weak Galerkin
algorithm (\ref{PD-WG32-1})-(\ref{PD-WG32-2}) is computed in several norms detailed as follows:

\begin{alignat*}{2}
\|v\|  & \ := \|v_0\|_{L^2(\Omega)},\qquad &&\mbox{($L^2$-norm)},\\
\|v\|_{L^1} & \ := \|v_0\|_{L^1(\Omega)},\qquad && \mbox{($L^1$-norm)},\\
\|v\|_{L^\infty} & \ := \|v_0\|_{L^\infty(\Omega)},\qquad && \mbox{($L^\infty$-norm)},\\
\|v\|_{1,h} & \ := \left( \sum_{T\in\T_h} h_T \|v_n\|_{0,\pT}^2
\right)^{\frac12},\qquad  &&  \mbox{(discrete $H^1$-norm)},\\
\|v\|_{W^{1,1}} & \ := \sum_{T\in\T_h} h_T \|v_n\|_{L^1(\pT)},\qquad  && \mbox{(discrete $W^{1,1}$-norm)},\\
\| v \|_{2,h} & \ := \left(\|\Delta v_0\|^2 +s(v,v)\right)^{\frac12},
\qquad && \mbox{(discrete $H^2$-norm)}.
\end{alignat*}

In our numerical experiments, the load function $f=f(x,y)$ and the Cauchy boundary data in the model problem (\ref{EQ:model-problem}) are computed according to the given exact solution $u=u(x,y)$. The uniform triangular partitions are obtained by first partitioning the domain $\Omega$ into $n\times n$ uniform sub-squares and then dividing each square element into two triangles by the diagonal line with negative slope.



Tables \ref{NE:TRI:Case1-0}-\ref{NE:TRI:Case1-3-2} illustrate the
performance of the numerical scheme when the boundary conditions are set as follows: (1) both Dirichlet and Neumann boundary conditions on the boundary segments $(0,1)\times 0$ and $1\times (0,1)$, (2) Dirichlet boundary condition on the boundary segment $0\times (0,1)$, and (3) Neumann boundary condition on $(0,1)\times 1$. The error between the numerical solution $u_h$ and the $L^2$ projection of the exact solution $Q_hu$ is denoted by $e=u_h-Q_hu$ (i.e., error function). The error function was measured in several norms, including the $L^1$, $L^\infty$, and $W^{1,1}$ for which no theory was developed in the previous section.

Table \ref{NE:TRI:Case1-0} demonstrates the correctness and reliability of the code with data from the exact solution $u=x^2+y^2-10xy$. It should be pointed out that the primal-dual weak Galerkin algorithm (\ref{PD-WG32-1})-(\ref{PD-WG32-2}) is exact when the exact solution is given as a quadratic polynomial. It can be seen from the table that the error is indeed in machine accuracy, particularly for relatively coarse grids. The computational results are thus in good consistency with the theory. This table gives us a great confidence on the correctness of the code implementation for the algorithm (\ref{PD-WG32-1})-(\ref{PD-WG32-2}). But it should be noted that the error seems to deteriorate when the mesh gets finer and finer. We conjecture that this deterioration might be caused by two factors: (1) the ill-posedness of the elliptic Cauchy problem, and (2) the poor conditioning of the discrete linear system.

Tables \ref{NE:TRI:Case1-1-1} - \ref{NE:TRI:Case1-3-2} show the numerical results when the exact solutions are given by $u=\sin(x)\sin(y)$,
$u=\cos(x)\cos(y)$, and $u=30xy(1-x)(1-y)$, respectively. All these
numerical results show that the convergence rate for the solution
of the primal-dual weak Galerkin algorithm (\ref{PD-WG32-1})-(\ref{PD-WG32-2}) in the discrete $H^2$
norm is of order $O(h)$, which is in great consistency with the
theory established in the previous sections. For the approximation
of $u_0$, the convergence rates in the usual $L^2$ norm, $L^1$ norm,
and $L^{\infty}$ norm seem to arrive at the order of $O(h^2)$. For the approximation of $u_n$ (i.e., the flux on element boundaries), the numerical rate of convergence is also at $O(h^2)$.
\begin{table}[H]
\begin{center}
 \caption{Numerical error and order of convergence for the exact solution $u=x^2+y^2-10xy$.}\label{NE:TRI:Case1-0}
 \begin{tabular}{|c|c|c|c|c|c|c|}
\hline $1/h$        &   $\| e\|_{2, h} $       & $\|e
\|_{L^1}$   & $\|e\|$ &$\|e\|_{1,h} $ & $\|e\|_{L^\infty}$
& $\|e\|_{W^{1,1}}$
\\
\hline 1&2.986E-14 & 5.877E-15&3.143E-15&1.948E-14 &
8.882E-15&2.917E-14
\\
\hline 2&6.168E-13&1.059E-14 &8.369E-15 &7.411E-14   &
3.741E-14&7.816E-14
\\
\hline 4 & 2.565E-12   &1.986E-14 &1.492E-14 &3.670E-13    &
7.550E-14 &4.051E-13
\\
\hline 8   &1.749E-11 &3.312E-14   & 2.438E-14 & 1.382E-12 &
1.849E-13     &1.307E-12
\\
\hline 16&  3.273E-10    &1.840E-13 & 1.578E-13 & 1.225E-11 &
1.651E-12 &   7.877E-12
\\
\hline 32  &4.799E-09    &8.080E-13 &   5.453E-13   & 8.206E-11
&6.134E-12  &   3.967E-11
\\
\hline
\end{tabular}
\end{center}
\end{table}

\begin{table}[H]
\begin{center}
\caption{Numerical error and order of convergence for the exact
solution $u=\sin(x)\sin(y)$.}\label{NE:TRI:Case1-1-1}
 \begin{tabular}{|c|c|c|c|c|c|c|}
\hline $1/h$ &  $\| e\|_{2, h} $ & order&
$\|e\|_{L^1}$   & order  & $\|e\|$  & order
\\
\hline 1   &0.1526 &&0.003771 &&   0.003107&
\\
\hline 2   &0.09246& 0.7227  &0.001798 &   1.068&  0.001331& 1.223
\\
\hline 4   &0.04283 &1.110 &0.0005520 &1.704 &0.0003955&   1.751
\\
\hline
8 & 0.01928 &1.152& 0.0001417 & 1.962&  0.0001011   &1.968  \\
\hline 16 &0.009002&   1.099& 3.521E-05&   2.009 &2.509E-05&   2.011
\\
\hline 32 &0.004343    &1.052 &8.731E-06   &2.012 & 6.226E-06 &
2.010
\\
\hline
\end{tabular}\end{center}
\end{table}

\begin{table}[H]
\begin{center}
\caption{Numerical error and order of convergence for the exact
solution $u=\sin(x)\sin(y)$.}\label{NE:TRI:Case1-1-2}
 \begin{tabular}{|c|c|c|c|c|c|c|}
\hline $1/h$ &  $\|e\|_{1,h}$ & order&   $\|e\|_{L^\infty}$ & order  &
$\|e\|_{W^{1,1}}$  & order
\\
\hline 1 & 0.07546 && 0.008561  && 0.1026  &
\\
\hline 2&  0.02313& 1.706 & 0.004312 & 0.9894&  0.02951 & 1.798
\\
\hline 4&  0.005580 & 2.051 & 0.001672 & 1.367&0.006976 &  2.081
\\
\hline 8&  0.001306& 2.096& 0.0004418&1.920 &0.001616&2.110
\\
\hline 16 &     0.0003010 & 2.075&  0.0001107&1.997&
0.00038177&2.081
\\
\hline 32&  7.512E-05 &2.045 &2.727E-05    &2.021 &9.264E-05&2.043
\\
\hline
\end{tabular}\end{center}
\end{table}

\begin{table}[H]
\begin{center}
\caption{Numerical error and order of convergence for the exact
solution $u=\cos(x)\cos(y)$.}\label{NE:TRI:Case1-2-1}
 \begin{tabular}{|c|c|c|c|c|c|c|}
\hline $1/h$ &  $\|e\|_{2, h} $ & order&
$\|e\|_{L^1}$   & order  & $\|e\|$  & order
\\
\hline 1& 0.1105&& 0.005454&&  0.005977&
\\
\hline 2 &0.07575 &0.5443&0.002730&0.9985& 0.002453&1.285
\\
\hline 4&0.03590 &1.077 &0.0005984 &2.190&0.0005052&   2.280
\\
\hline 8&0.01701 &1.078 &0.0001345 &2.153 &0.0001069  &2.240
\\
\hline 16 &0.008379 &1.022&3.164E-05& 2.088&2.460E-05&2.120
\\
\hline
32 &0.004183&1.002  &7.662E-06  &2.046 &5.912E-06& 2.057\\
\hline
\end{tabular}\end{center}
\end{table}

\begin{table}[H]
\begin{center}
\caption{Numerical error and order of convergence for the exact
solution $u=\cos(x)\cos(y)$.}\label{NE:TRI:Case1-2-2}
 \begin{tabular}{|c|c|c|c|c|c|c|}
\hline $1/h$ &  $\|e\|_{1,h}$ & order&   $\|e\|_{L^\infty}$ & order  & $\|e\|_{W^{1,1}}$  & order
\\
\hline 1&0.1373    &&0.02067   &&  0.2133&
\\
\hline
 2&0.03903 &1.814 &0.01073 &0.9461&0.04813  &2.148
\\
\hline 4&0.008740 &2.159&0.002808  &1.934 &0.01039 &2.212
\\
\hline 8&0.001950 &2.164 &0.0006421 &2.129&0.002365 &2.135
\\
\hline 16 &0.0004558 & 2.097 &0.0001550 &2.051&0.0005639 &2.068
\\
\hline 32& 0.0001103&2.047 &3.844E-05  &2.011 &    0.0001379
&2.032
\\
\hline
\end{tabular}
\end{center}
\end{table}

\begin{table}[H]
\begin{center}
\caption{Numerical error and order of convergence for the exact
solution $u=30xy(1-x)(1-y)$.}\label{NE:TRI:Case1-3-1}
 \begin{tabular}{|c|c|c|c|c|c|c|}
\hline $1/h$ &  $\|e\|_{2, h} $ & order&
$\|e\|_{L^1}$   & order  & $\|e\|$  & order
\\
\hline 1 & 21.37&& 0.4892  &&  0.4641  &
\\
\hline 2 & 9.937 &1.104 & 0.2044 &1.259 &0.1871 &1.311
\\
\hline 4 & 4.487 &1.147 &0.04177 &2.291&0.03583 & 2.385
\\
\hline 8 & 2.100 &1.095 &0.009281 &2.170 &0.007429 &2.270
\\
\hline 16 & 1.026 &1.032 &0.002187 &2.085 &0.001709  &2.120
\\
\hline 32 & 0.5105 &1.008 &0.0005300 &2.045 &0.0004113  &2.056
\\
\hline
\end{tabular}\end{center}
\end{table}

\begin{table}[H]
\begin{center}
\caption{Numerical error and order of convergence for the exact
solution $u=30xy(1-x)(1-y)$.}\label{NE:TRI:Case1-3-2}
 \begin{tabular}{|c|c|c|c|c|c|c|}
\hline $1/h$ &  $\|e\|_{1,h}$ & order&   $\|e\|_{L^\infty}$ & order  & $\|e\|_{W^{1,1}}$  & order \\
\hline 1& 16.12 && 1.557 &&18.11 &
\\
\hline 2& 4.138&1.962  & 0.8164    & 0.9318&   4.647   & 1.963
\\
\hline 4& 0.9333& 2.149& 0.2066& 1.982     & 1.055 & 2.139
\\
\hline
8& 0.2120& 2.138 & 0.04629& 2.158   & 0.2439    & 2.113\\
\hline 16& 0.05007 & 2.082& 0.01106& 2.066& 0.05865& 2.056
\\
\hline 32& 0.01217 & 2.041& 0.002753& 2.006 &  0.01439& 2.027
\\
\hline
\end{tabular}
\end{center}
\end{table}


Tables \ref{NE:TRI:Case2-1-1}-\ref{NE:TRI:Case2-3-2} demonstrate the
performance of the PD-WG algorithm when the boundary conditions are set as follows: (1) Dirichlet on the boundary segments $(0,1)\times 0$ and $0\times (0,1)$, and (2) Neumann on the boundary segments
$1\times (0,1)$ and  $(0,1)\times 1$. Note that this is a standard mixed boundary value problem, and no Cauchy data is given on the boundary. The purpose of this test is to show the efficiency of the PD-WG algorithm (\ref{PD-WG32-1})-(\ref{PD-WG32-2}) for classical well-posed problems.

Tables \ref{NE:TRI:Case2-1-1}-\ref{NE:TRI:Case2-3-2} show the
numerical results when the exact solutions are given by
$u=\sin(x)\sin(y)$, $u=\cos(x)\cos(y)$ and $u=30xy(1-x)(1-y)$,
respectively. The numerical results show that the convergence
for the solution of the primal-dual weak Galerkin algorithm (\ref{PD-WG32-1})-(\ref{PD-WG32-2}) is of order $O(h)$ in
the discrete $H^2$ norm, which is consistent with the established theory.
For the approximation of $u_0$, the convergence in the usual
$L^2$ norm, $L^1$ norm,  and $L^{\infty}$ norm is of the
order $2$. Regarding the approximation of $u_n$, the tables show that it also converges at the rate of $O(h^2)$, as measured by
the $H^1$ norm and $W^{1,1}$ norm for the error function $e=u_h-Q_h u$.

\begin{table}[H]
\begin{center}
\caption{Numerical error and order of convergence for the exact
solution $u=\sin(x)\sin(y)$.}\label{NE:TRI:Case2-1-1}
 \begin{tabular}{|c|c|c|c|c|c|c|}
\hline $1/h$ &  $\|e\|_{2, h} $ & order&
$\|e\|_{L^1}$   & order  & $\|e\|$  & order
\\
\hline 1 & 0.1530& & 0.04214   & &     0.02633 &
\\
\hline 2 & 0.09251 & 0.7254 & 0.009863& 2.095 & 0.006575   & 2.002
\\
\hline 4 & 0.04197& 1.140 & 0.002260& 2.126& 0.001546& 2.089
\\
\hline 8 & 0.01908& 1.137 & 0.0005295 & 2.093 & 0.0003674&  2.073
\\
\hline 16 & 0.008969 &1.089& 0.0001275 & 2.054& 8.943E-05&     2.038
\\
\hline 32 & 0.004336 &1.048 &3.127E-05&2.028& 2.207E-05&   2.019
\\
\hline
\end{tabular}\end{center}
\end{table}

\begin{table}[H]
\begin{center}
\caption{Numerical error and order of convergence for the exact
solution $u=\sin(x)\sin(y)$.}\label{NE:TRI:Case2-1-2}
 \begin{tabular}{|c|c|c|c|c|c|c|}
\hline $1/h$ &  $\|e\|_{1,h}$ & order&   $\|e\|_{L^\infty}$ & order  &
$\|e\|_{W^{1,1}}$  & order
\\
\hline 1&  0.03160 &&  0.06993 &&  0.04888 &
\\
\hline 2&   0.02294 &0.4622 &0.02210 &1.662&0.03219&   0.6025
\\
\hline 4&   0.005248 &2.128&0.005788 &1.933 &0.006907  &2.221
\\
\hline 8&   0.001134 &2.210&0.001446 &2.001&0.001475   &2.227
\\
\hline 16&  0.0002570 &2.142&0.0003619&1.999&0.0003446 &2.098
\\
\hline 32&  6.088E-05&2.078 &9.065E-05&1.997   &8.339E-05& 2.047
\\
\hline
\end{tabular}\end{center}
\end{table}

\begin{table}[H]
\begin{center}
\caption{Numerical error and order of convergence for the exact
solution $u=\cos(x)\cos(y)$.}\label{NE:TRI:Case2-2-1}
 \begin{tabular}{|c|c|c|c|c|c|c|}
\hline $1/h$ &  $\| e\|_{2, h} $ & order&
$\|e\|_{L^1}$   & order  & $\|e\|$  & order
\\
\hline
1 & 0.04831 &&0.07023 && 0.04350 &\\
\hline 2& 0.06093 &-0.3348 &0.009188& 2.934 & 0.006343     &2.778
\\
\hline
4&0.03252&  0.9060  &0.001543   &2.574& 0.001142 &  2.473 \\
\hline 8 & 0.01637 &   0.9897  &0.0003096  &2.318&0.0002416
&2.241
\\
\hline 16 & 0.008254   &0.9883 & 6.91403E-05   &2.1623 & 5.554E-05
&2.121
\\
\hline 32& 0.004155    &0.9900 &   1.633E-05&  2.082&  1.330E-05&
2.062
\\
\hline
\end{tabular}\end{center}
\end{table}

\begin{table}[H]
\begin{center}
\caption{Numerical error and order of convergence for the exact
solution $u=\cos(x)\cos(y)$.}\label{NE:TRI:Case2-2-2}
 \begin{tabular}{|c|c|c|c|c|c|c|}
\hline $1/h$ &  $\|e\|_{1,h}$ & order&   $\|e\|_{L^\infty}$ & order  &
$\|e\|_{W^{1,1}}$  & order
\\
\hline 1& 0.2011 &&    0.1059 &&   0.3093  &
\\
\hline 2& 0.04311 &    2.222&  0.01900 &2.479& 0.04857 &   2.671
\\
\hline 4& 0.009095 &   2.245&  0.004171    &2.188& 0.01019&2.254
\\
\hline 8& 0.001941&    2.228 & 0.001027    &2.021 &0.002217 &  2.120
\\
\hline 16& 0.0004384&  2.147&  0.0002556   &2.007& 0.0005157 &2.104
\\
\hline 32& 0.0001039&  2.077 & 6.363E-05   &2.006 &0.0001243
&2.052
\\
\hline
\end{tabular}
\end{center}
\end{table}

\begin{table}[H]
\begin{center}
\caption{Numerical error and order of convergence for the exact
solution $u=30xy(1-x)(1-y)$.}\label{NE:TRI:Case2-3-1}
 \begin{tabular}{|c|c|c|c|c|c|c|}
\hline $1/h$ &  $\| e\|_{2, h} $ & order&
$\|e\|_{L^1}$   & order  & $\|e\|$  & order
\\
\hline 1&18.87 &&  4.443 &&    2.749   &
\\
\hline 2 & 9.260 &1.027&0.5741&2.952&0.3994&2.783
\\
\hline 4& 4.272&1.116 &0.09675&2.569 & 0.07174     &2.477
\\
\hline 8 & 2.065&1.049&0.01927&2.3281  &0.01511    &2.247
\\
\hline 16 & 1.022&1.015&0.004261 &2.177&0.003451&2.131
\\
\hline 32 & 0.5099&1.003&0.0009992&2.092   &0.0008205& 2.072
\\
\hline
\end{tabular}\end{center}
\end{table}

\begin{table}[H]
\begin{center}
\caption{Numerical error and order of convergence for the exact
solution $u=30xy(1-x)(1-y)$.}\label{NE:TRI:Case2-3-2}
 \begin{tabular}{|c|c|c|c|c|c|c|}
\hline $1/h$ &  $\|e\|_{1,h}$ & order&   $\|e\|_{L^\infty}$ & order  &
$\|e\|_{W^{1,1}}$  & order
\\
\hline
1& 17.60    &&  6.429   && 23.57    & \\
\hline 2& 4.154&2.083&1.118&2.524 &4.759  & 2.308
\\
\hline 4& 0.91502 &2.183& 0.2643 & 2.080 &1.072 &2.150
\\
\hline 8& 0.2011 &2.186&   0.06844     &1.949 &    0.2432  &2.140
\\
\hline 16& 0.04598 &2.129  &0.01731    &1.983  &0.05736 &2.084
\\
\hline 32& 0.01093 &2.073 &    0.004323&2.002&0.01392& 2.043
\\
\hline
\end{tabular}
\end{center}
\end{table}


Table \ref{NE:TRI:Case3-1-1} demonstrates the performance of the
PD-WG algorithm (\ref{PD-WG32-1})-(\ref{PD-WG32-2}) when the boundary conditions are set as follows: (1) Dirichlet and Neumann on the boundary segment $(0,1)\times 0$, (2) Dirichlet on the boundary segment $0\times (0,1)$, and Neumann on the boundary segment $(0,1)\times 1$. Table \ref{NE:TRI:Case3-1-1}
shows the numerical results when the exact solution is given by
$u=\sin(x)\sin(y)$.  These numerical results demonstrate that the
convergence for the primal-dual weak Galerkin solution in the discrete $H^2$ norm is of order $O(h)$, which is in perfect consistency with the theory. For the approximation of $u_0$, the numerical convergence in
the usual $L^2$ and $L^1$ norms arrives at a rate faster than $O(h^2)$.

\begin{table}[H]
\begin{center}
\caption{Numerical error and order of convergence for the exact
solution $u=\sin(x)\sin(y)$.}\label{NE:TRI:Case3-1-1}
 \begin{tabular}{|c|c|c|c|c|c|c|}
\hline $1/h$ &  $\| e\|_{2, h} $ & order&
$\|e\|_{L^1}$   & order  & $\|e\|$  & order
\\
\hline 1 &  0.1564     &&  0.05116     &&  0.03204 &
\\
\hline 2 &  0.09258    &0.7565 &0.01478    &1.791 &    0.01009
&1.667
\\
\hline 4 &  0.04196&   1.142 &0.003668 &2.011& 0.002651 &1.928
\\
\hline
8 &  0.01877 &  1.160 & 0.0007335  &2.322&  0.0005444 &2.284 \\
\hline 16 &  0.008872 &1.081 & 0.0001421 & 2.368& 0.0001075 &2.340
\\
\hline 32 &  0.004313  &1.040 &    2.820E-05& 2.334& 2.226E-05&
2.272
\\
\hline
\end{tabular}\end{center}
\end{table}


Tables \ref{NE:TRI:Case4-1-1}-\ref{NE:TRI:Case4-1-2} demonstrate the
performance of the algorithm (\ref{PD-WG32-1})-(\ref{PD-WG32-2}) when the boundary conditions are set as the
following: the boundary segments $0\times (0,1)$ and $1\times (0,1)$ are given by both the Dirichlet and Neumann boundary conditions for the exact
solution $u=30xy(1-x)(1-y)$. The numerical results show that the convergence for the solution of the primal-dual weak Galerkin algorithm in the discrete $H^2$ norm is of order $O(h)$. For the approximation $u_0$, the
convergence in the usual $L^2$ norm, $L^1$ norm and $L^{\infty}$
norm is at the rate of $O(h^2)$. When it comes to $u_n$, the numerical convergence is clearly at the rate of $O(h^2)$, as shown in the discrete $H^1$ and $W^{1,1}$ norms in the table.

\begin{table}[H]
\begin{center}
\caption{Numerical error and order of convergence for the exact
solution $u=30xy(1-x)(1-y)$.}\label{NE:TRI:Case4-1-1}
 \begin{tabular}{|c|c|c|c|c|c|c|}
\hline $1/h$ &  $\| e\|_{2, h} $ & order&
$\|e_0\|_{L^1}$   & order  & $\|e\|$  & order
 \\
 \hline
1&   31.67  &&  0.625   &&  0.4419  &
 \\
 \hline
2&   10.16&1.641&0.3859&0.6958&0.2411   &0.8744  \\
 \hline
4&   3.937 &1.367 &0.1150 &1.746&0.07084&1.767
 \\
 \hline
8&  1.946 &1.016  &0.0287&2.003 &0.01809 &1.969
 \\
 \hline
16&  0.9935 &0.9703&0.006835 &2.070&0.004441    &2.026
 \\
 \hline
 32&     0.5027 &0.9827 &0.001651&2.050&0.001096 &2.018
  \\
 \hline
 \end{tabular}\end{center}
\end{table}

\begin{table}[H]
\begin{center}
\caption{Numerical error and order of convergence for the exact
solution $u=30xy(1-x)(1-y)$.}\label{NE:TRI:Case4-1-2}
 \begin{tabular}{|c|c|c|c|c|c|c|}
\hline $1/h$ &  $\|e\|_{1,h}$ & order&   $\|e\|_{L^\infty}$ & order  &
$\|e\|_{W^{1,1}}$  & order
\\
\hline 1& 10.75    &&  1.25    &&  13.75   &
\\
\hline 2& 2.766&1.959 &0.6134 &1.027 &3.63 &   1.920
\\
\hline 4& 0.7151&1.951 &0.1928     &1.670& 0.9166& 1.987
\\
\hline
8& 0.1827 &1.969&0.05771&1.740 &0.2318 &    1.983 \\
\hline
16& 0.04836&1.917 &0.01690 &1.771 &0.06130&1.919  \\
\hline 32& 0.01265&1.935 &0.004636 &1.866 &0.01592&1.945
 \\
 \hline
 \end{tabular}
\end{center}
\end{table}

Table \ref{NE:TRI:Case5-1-1} demonstrate the performance of the PD-WG algorithm (\ref{PD-WG32-1})-(\ref{PD-WG32-2}) when the boundary conditions are set as follows: the boundary
segment $(0,1)\times 0$ is given by both the Dirichlet and Neumann
boundary conditions for three exact solutions
$u_1=\sin(x)\sin(y)$, $u_2=\cos(x)\cos(y)$, and $u_3=30xy(1-x)(1-y)$. All
these numerical results illustrate that the convergence for the
solution of the primal-dual weak Galerkin algorithm in the discrete
$H^2$ norm is at the rate of $O(h)$. This is in great consistency with
the theory established in the previous sections.

\begin{table}[H]
\begin{center}
\caption{Numerical error and order of convergence for the exact
solutions $u_1=\sin(x)\sin(y)$, $u_2=\cos(x)\cos(y)$ and
$u_3=30xy(1-x)(1-y)$.}\label{NE:TRI:Case5-1-1}
 \begin{tabular}{|c|c|c|c|c|c|c|}
\hline $1/h$ & $\| e\|_{2, h}$ for $u_1$  & order&   $\| e\|_{2,
h} $ for $u_2$  & order  & $\| e\|_{2, h} $ for $u_3$  & order
\\
\hline 1& 0.06734  &&  0.09480     &&  6.667   &
\\
\hline
2& 0.07936 &    -0.2371  &0.04827&0.9739 &6.1728&   0.1110 \\
\hline
4& 0.04555& 0.8012 &0.02475&0.9638 &    3.848   &0.6820\\
\hline
8& 0.02338&0.9622&0.01137 &1.122 &1.794 &   1.101 \\
\hline
16&0.01185&0.9807 &0.006011 &   0.9197  &0.8191     &1.131   \\
\hline 32&0.005911 &1.003 &0.003210 & 0.9051 &0.3870 &1.082
 \\
 \hline
 \end{tabular}
\end{center}
\end{table}

Our numerical experiments indicate that the numerical
performance of the primal-dual weak Galerkin finite element scheme
(\ref{PD-WG32-1})-(\ref{PD-WG32-2}) is typically better than what
the theory predicts. We feel that the primal-dual weak Galerkin finite element method is an efficient and reliable numerical method for
the ill-posed elliptic Cauchy problem.

Finally, Figures \ref{fig:mesh}-\ref{fig:mesh4} illustrate the surface plots of the PD-WG approximations together with the numerical Lagrange multipliers for a test case with exact solution $u=\cos(x) \cos(y)$ in mind. The Cauchy condition is imposed on the boundary segment $\Gamma_d=\Gamma_n=(0,0.5)\times 0$. Figure \ref{fig:mesh} shows the solution when the exact boundary data is employed. Figure \ref{fig:mesh2} shows the numerical solutions when the exact boundary data is perturbed by a random noise represented as $0.005*(0.5-Rand)$, where $Rand$ is the MatLab function that generates random numbers in the range $(0,1)$. The purpose of this numerical experiment is to see the sensitivity of the numerical scheme with respect to random noise on the boundary data. It can be seen that the PD-WG scheme works well when the Cauchy data is exact (which is the assumption of the present paper). On the other hand, the scheme seems to be very sensitive to even small perturbations on the boundary data. Readers are invited to draw their own conclusions from these plots. It is evident that a further study is necessary for elliptic Cauchy problems with noise on the boundary data.

\begin{figure}[H]
\centering
\subfigure[PD-WG approximation]{
\label{Fig.sub.2.lv}
\includegraphics [width=0.3\textwidth]{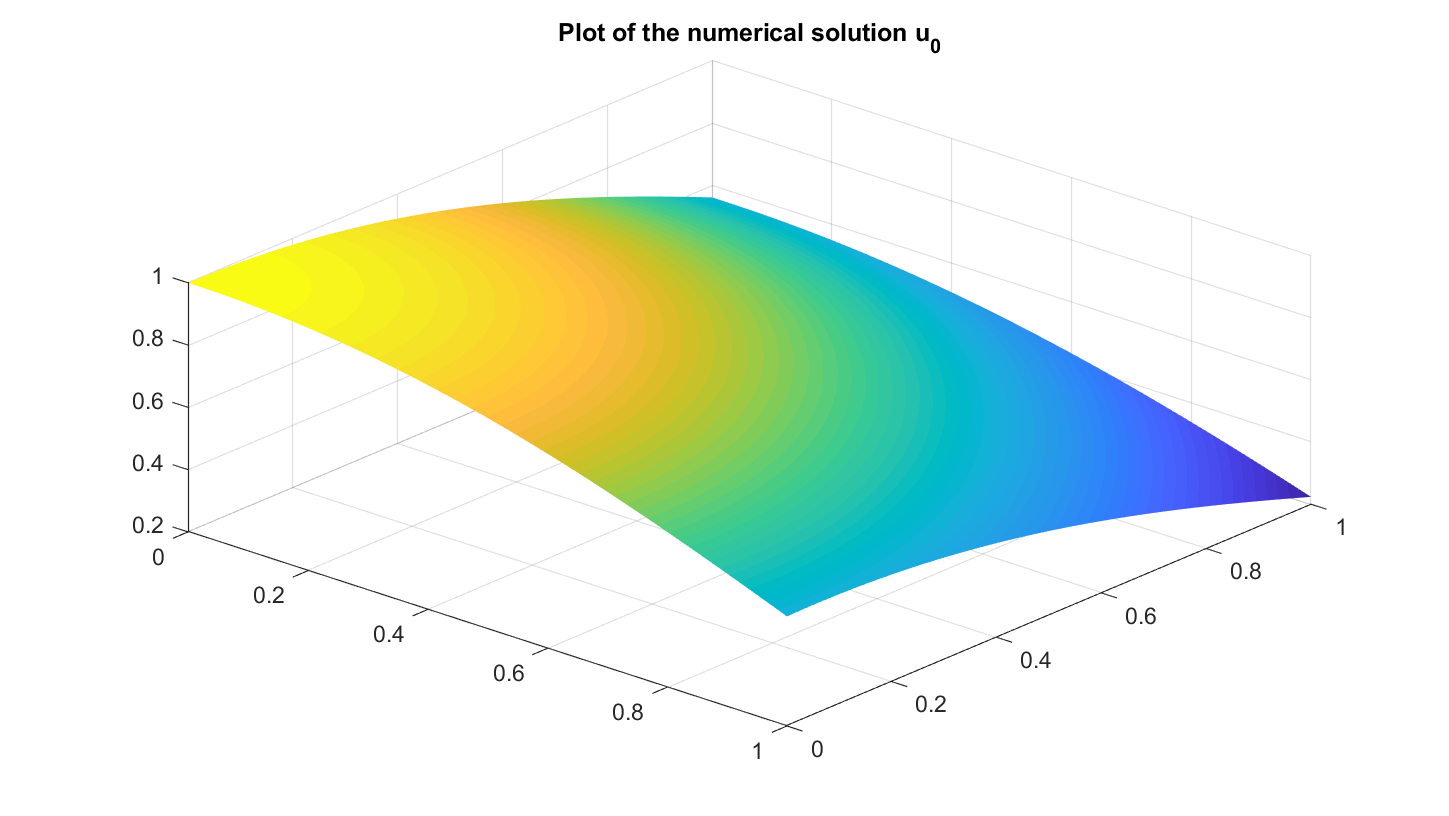}
}
\subfigure[Lagrange multiplier]{
\label{Fig.sub.2.2s}
\includegraphics [width=0.3\textwidth]{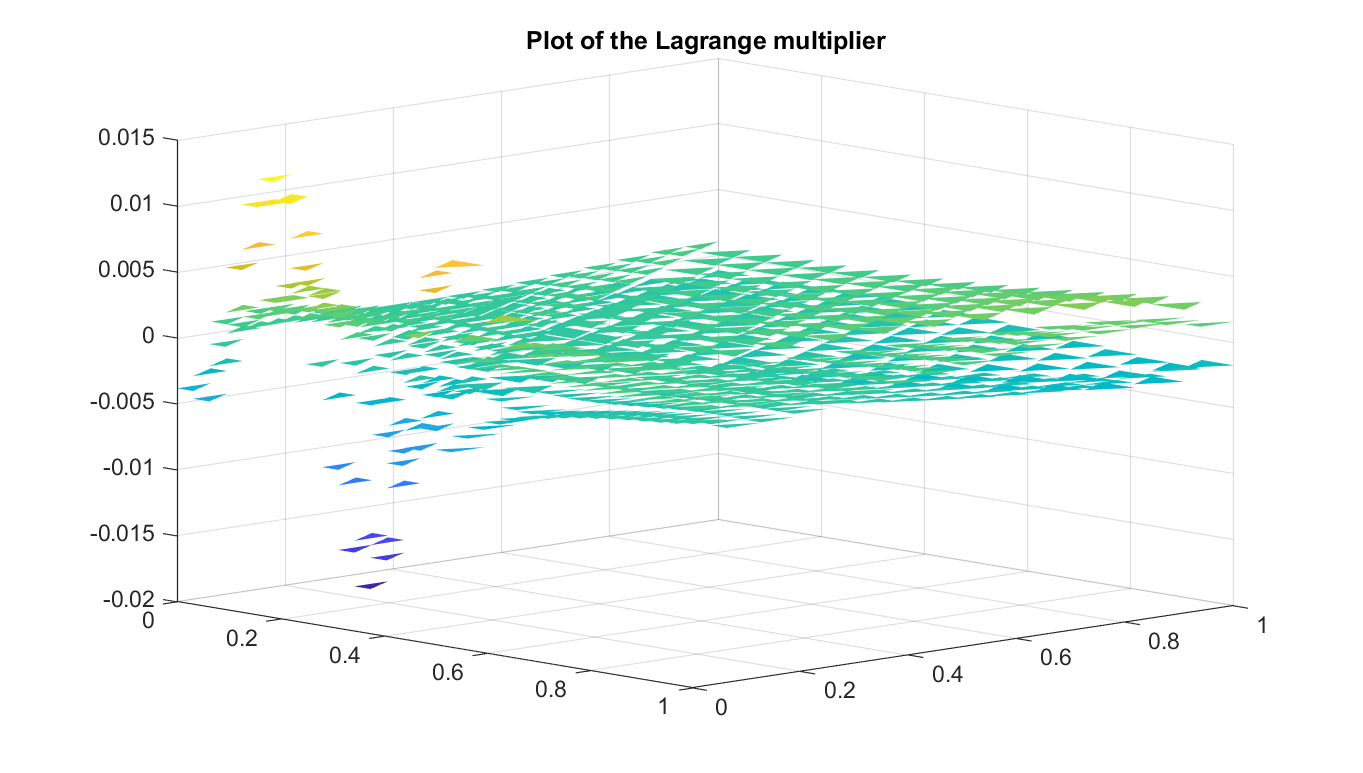}
}
\subfigure[Error function]{
\label{Fig.sub.2.3p}
\includegraphics [width=0.3\textwidth]{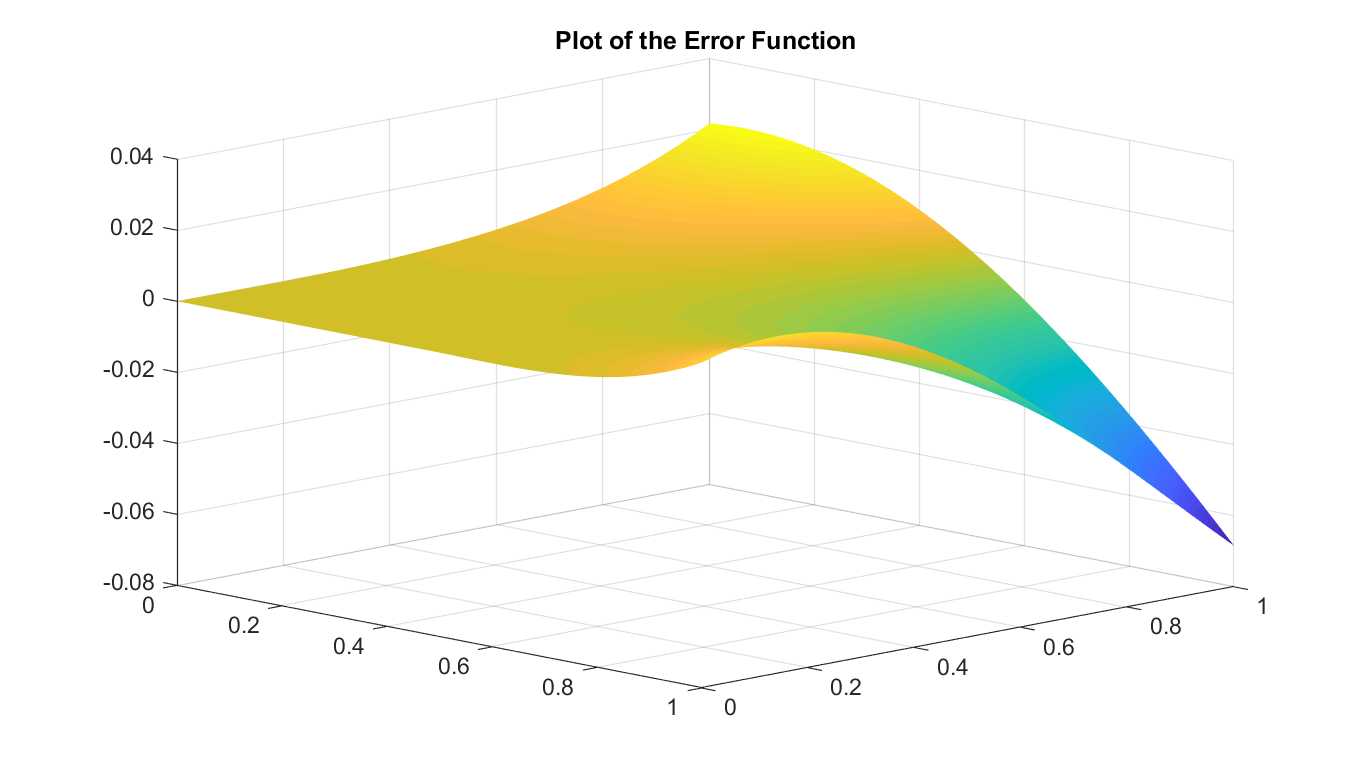}
}
\caption{Numerical results with exact Cauchy data.}
\label{fig:mesh}
\end{figure}

\begin{figure}[H]
\centering
\subfigure[PD-WG approximation]{
\label{Fig.sub.3.lv}
\includegraphics [width=0.3\textwidth]{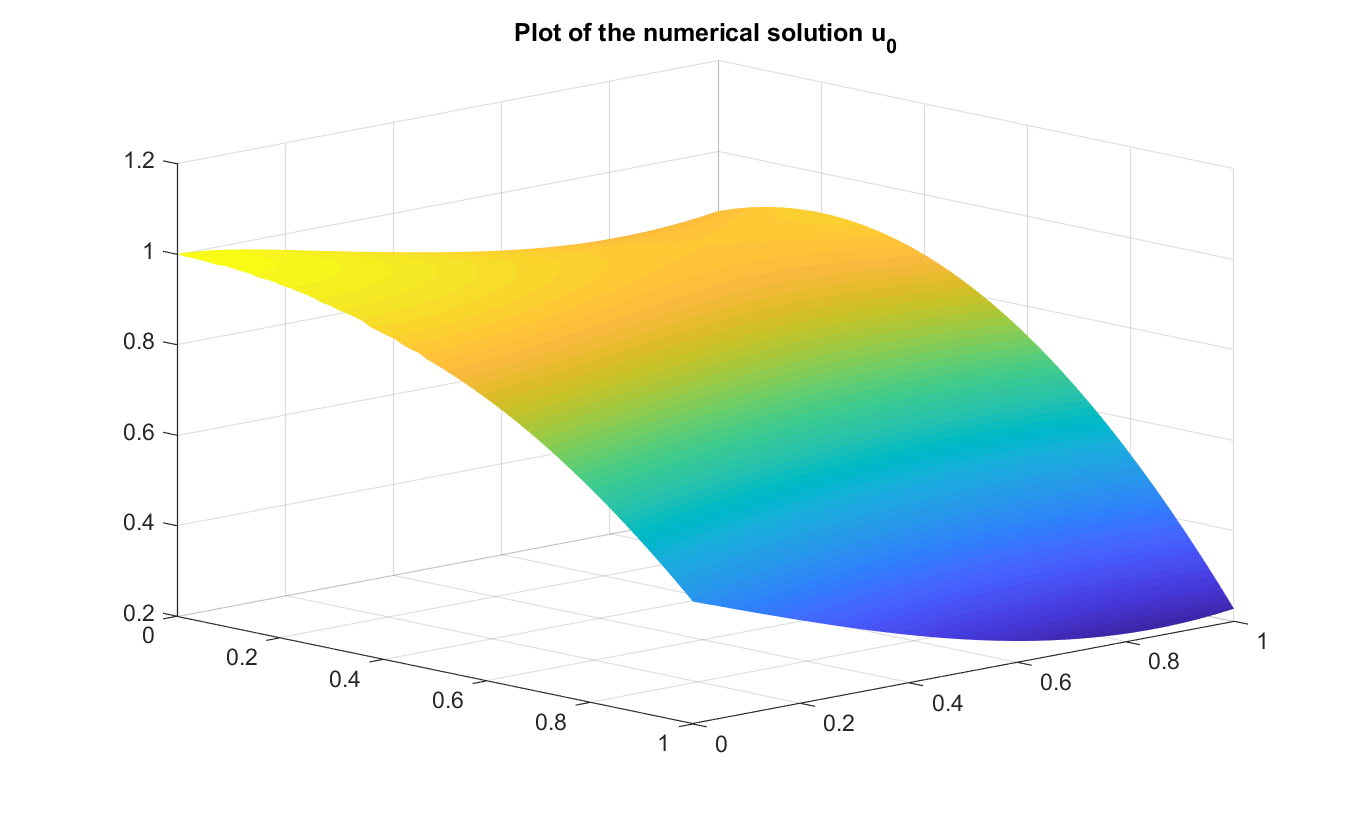}
}
\subfigure[Lagrange multiplier]{
\label{Fig.sub.3.2s}
\includegraphics [width=0.3\textwidth]{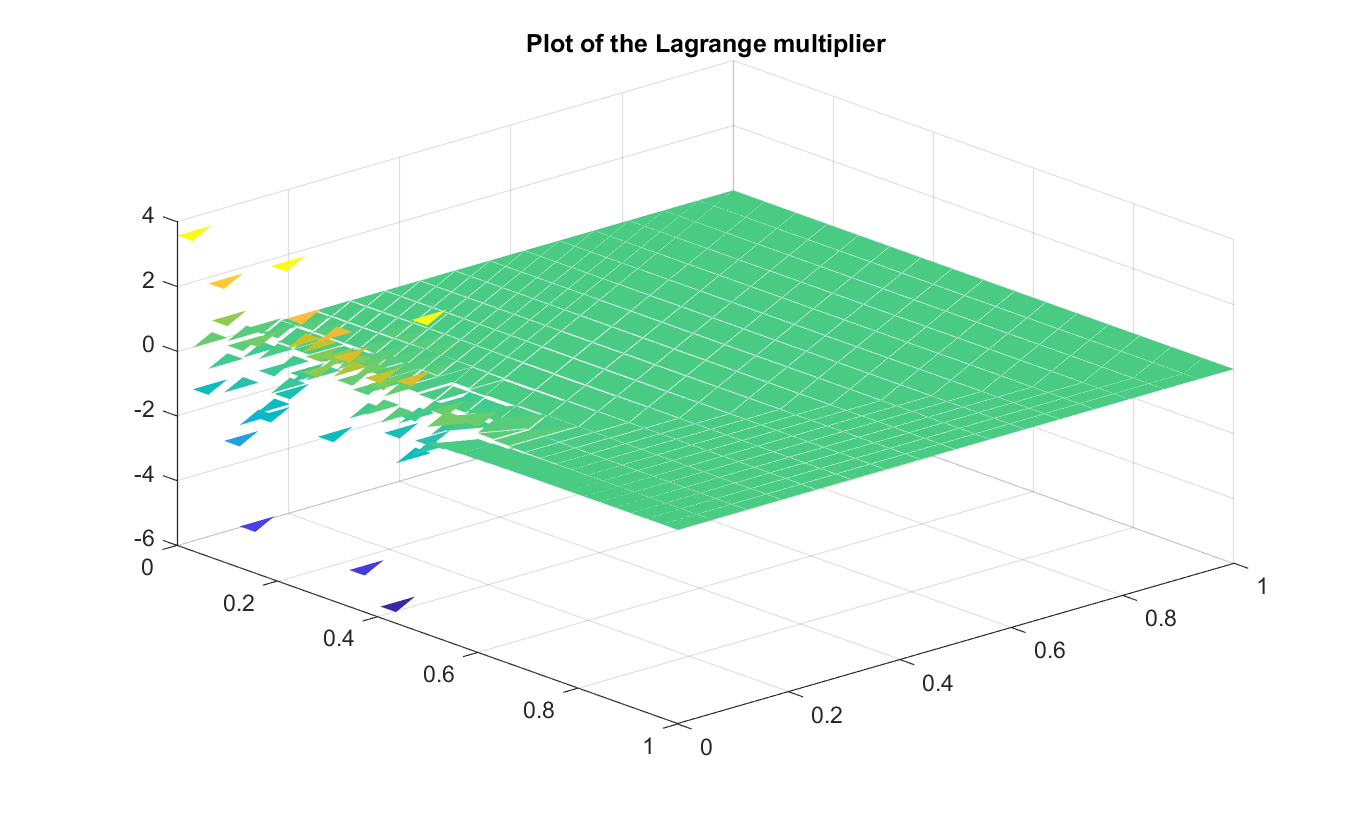}
}
\subfigure[Error function]{
\label{Fig.sub.3.3p}
\includegraphics [width=0.3\textwidth]{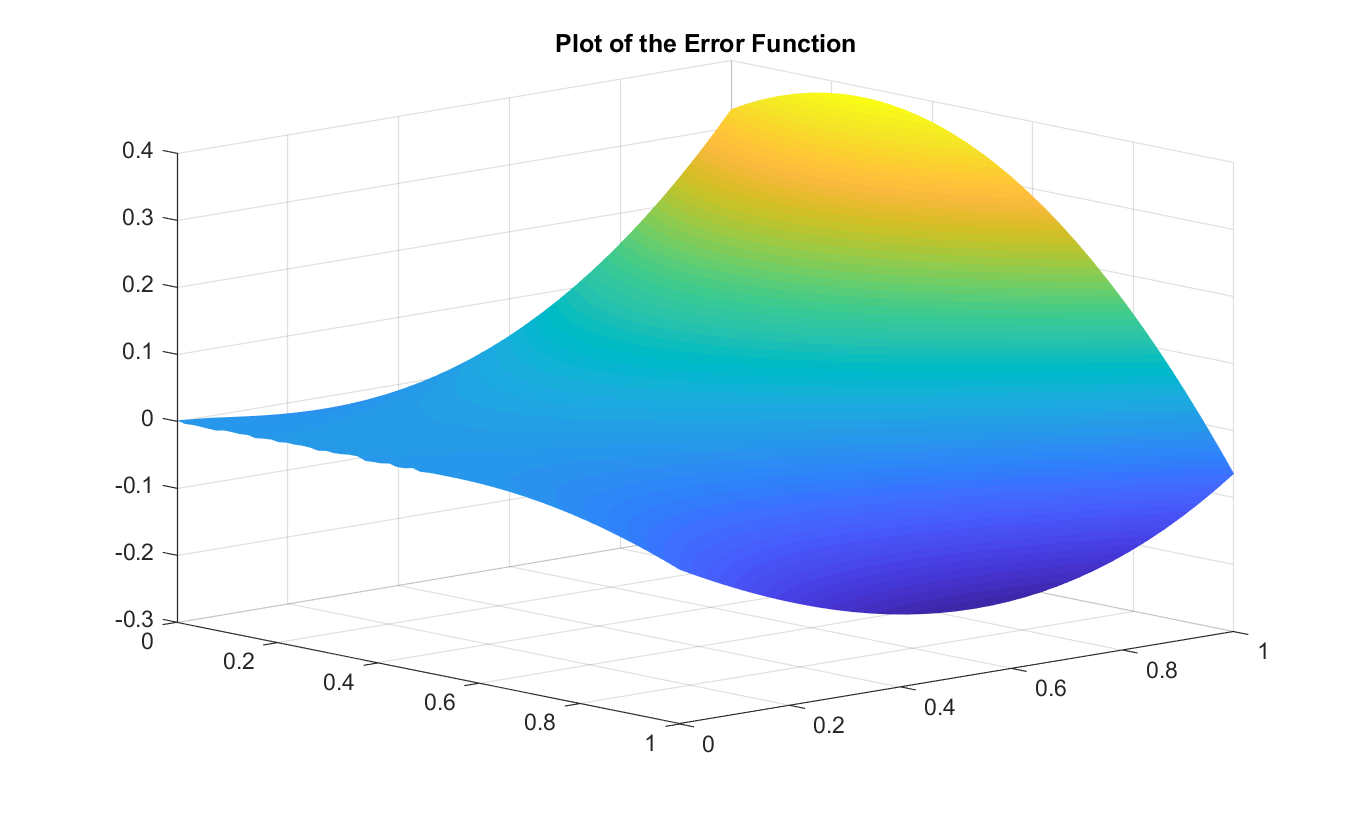}
}
\caption{Numerical results with a random perturbation by $0.005*(0.5-Rand)$ of the Cauchy data.}
\label{fig:mesh2}
\end{figure}

\begin{figure}[H]
\centering
\subfigure[PD-WG approximation]{
\label{Fig.sub.4.lv}
\includegraphics [width=0.3\textwidth]{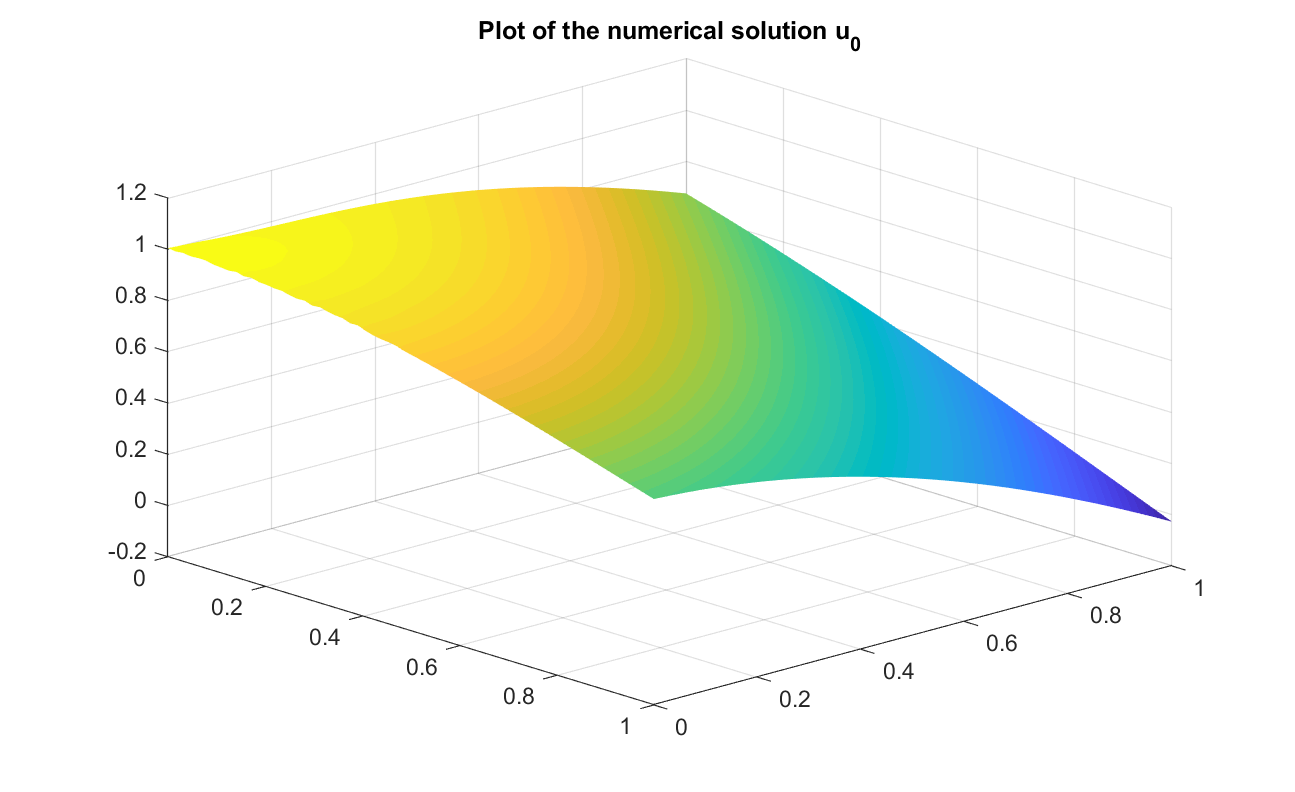}
}
\subfigure[Lagrange multiplier]{
\label{Fig.sub.4.2s}
\includegraphics [width=0.3\textwidth]{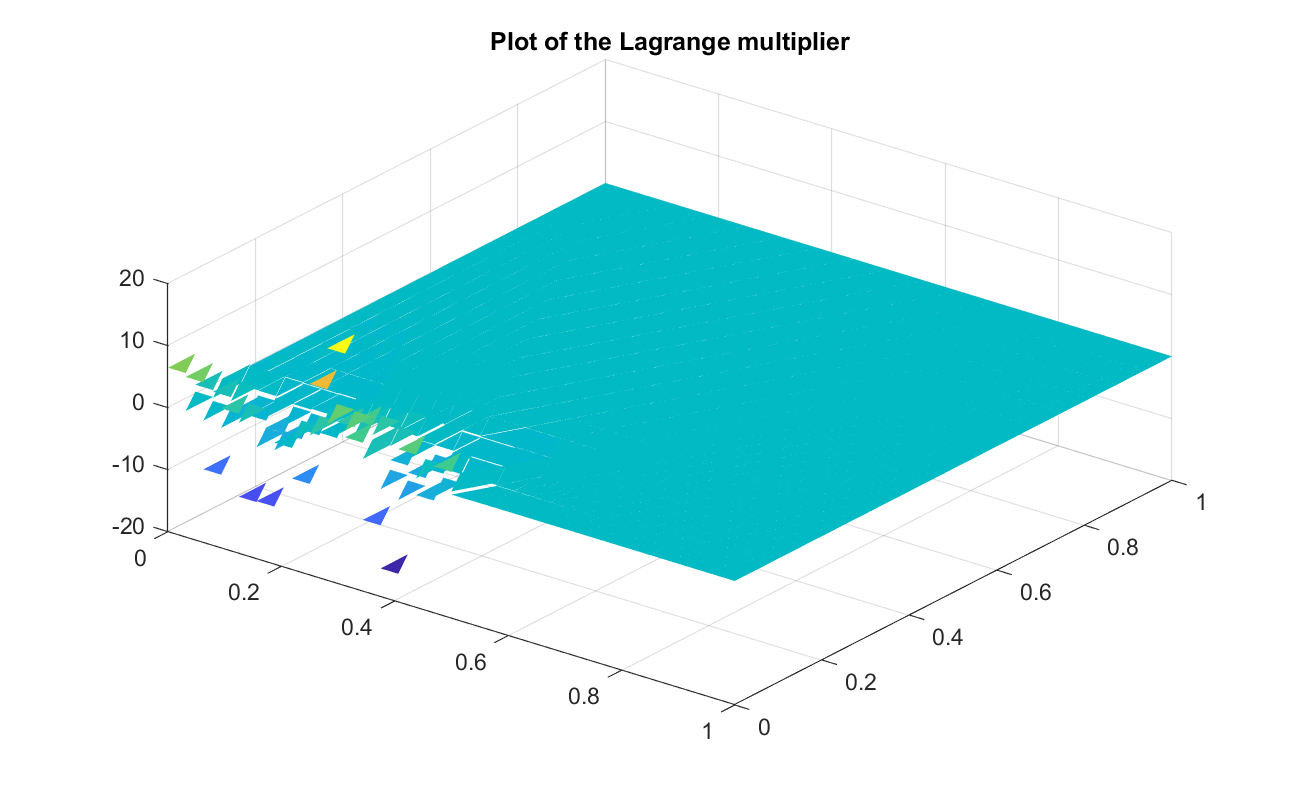}
}
\subfigure[Error function]{
\label{Fig.sub.4.3p}
\includegraphics [width=0.3\textwidth]{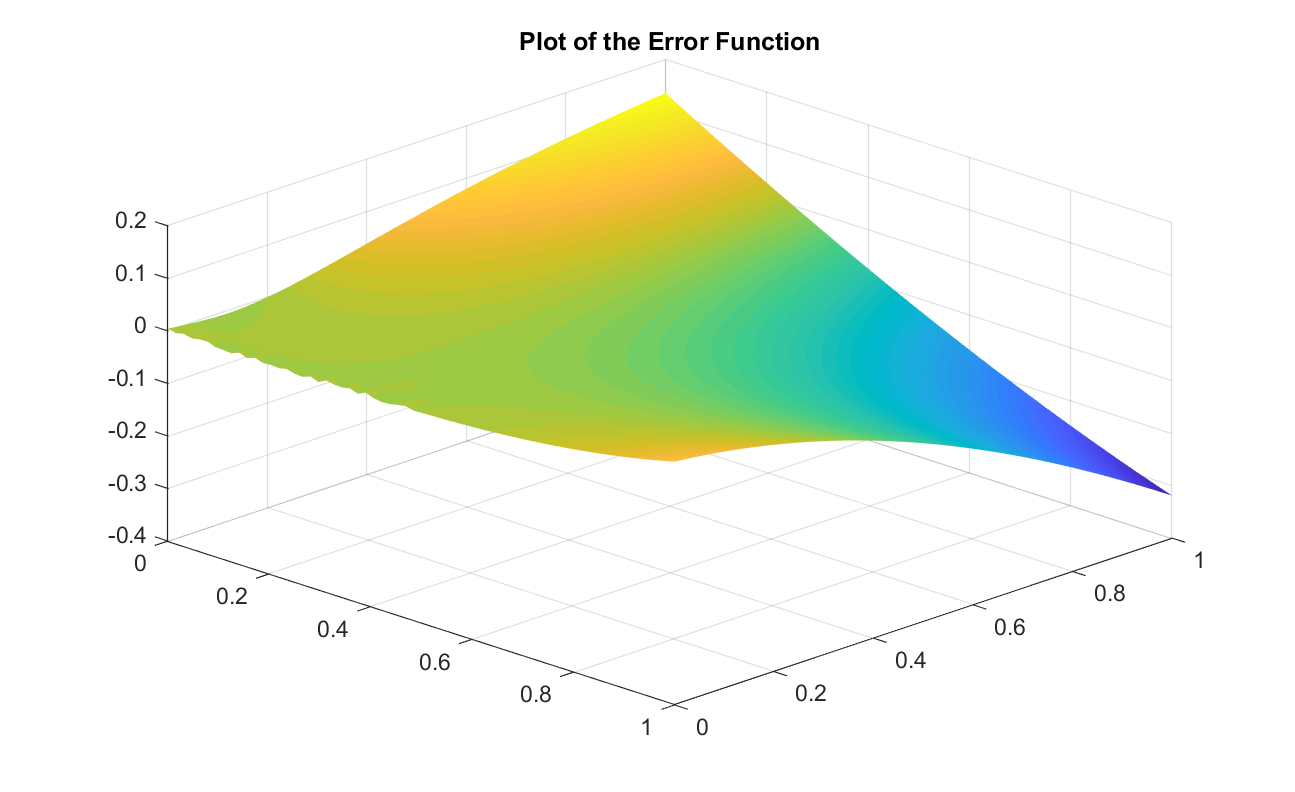}
}
\caption{Numerical results with a random perturbation by $0.01*(0.5-Rand)$ of the Cauchy data.}
\label{fig:mesh3}
\end{figure}

\begin{figure}[H]
\centering
\subfigure[PD-WG approximation]{
\label{Fig.sub.5.lv}
\includegraphics [width=0.3\textwidth]{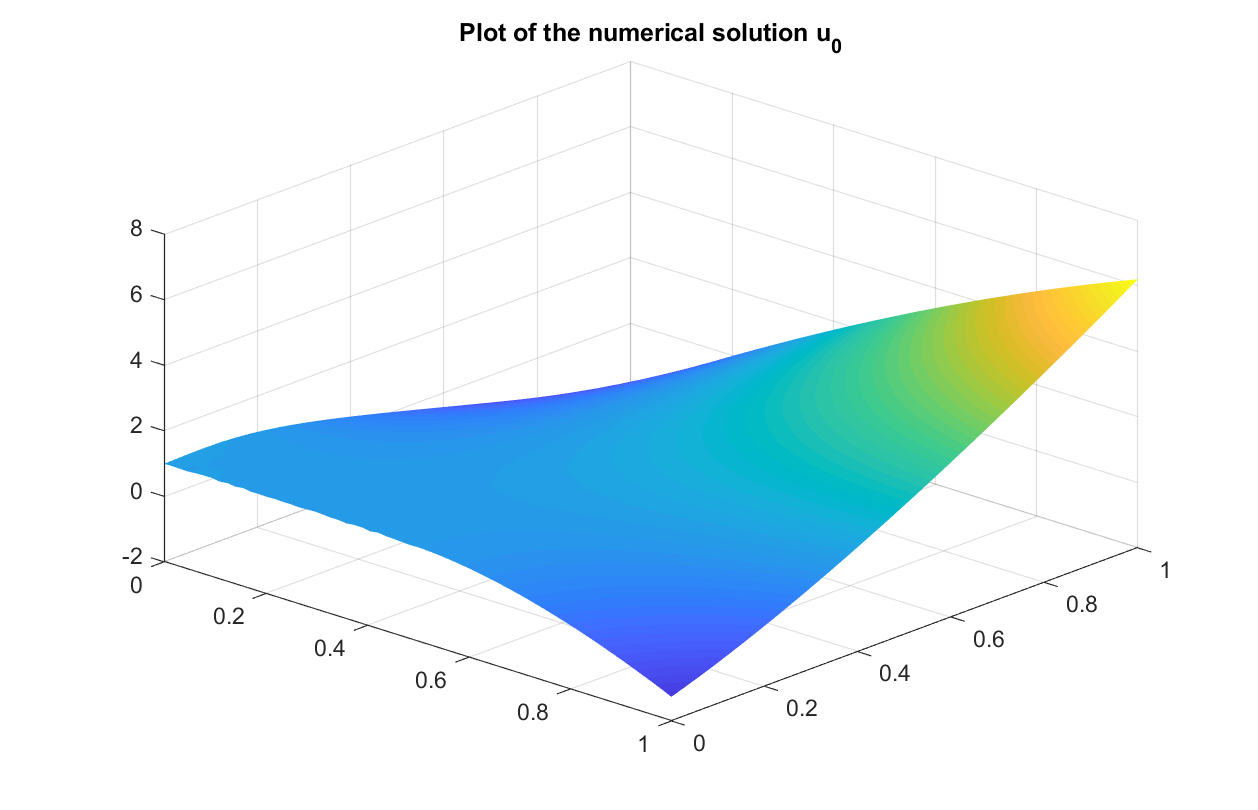}
}
\subfigure[Lagrange multiplier]{
\label{Fig.sub.5.2s}
\includegraphics [width=0.3\textwidth]{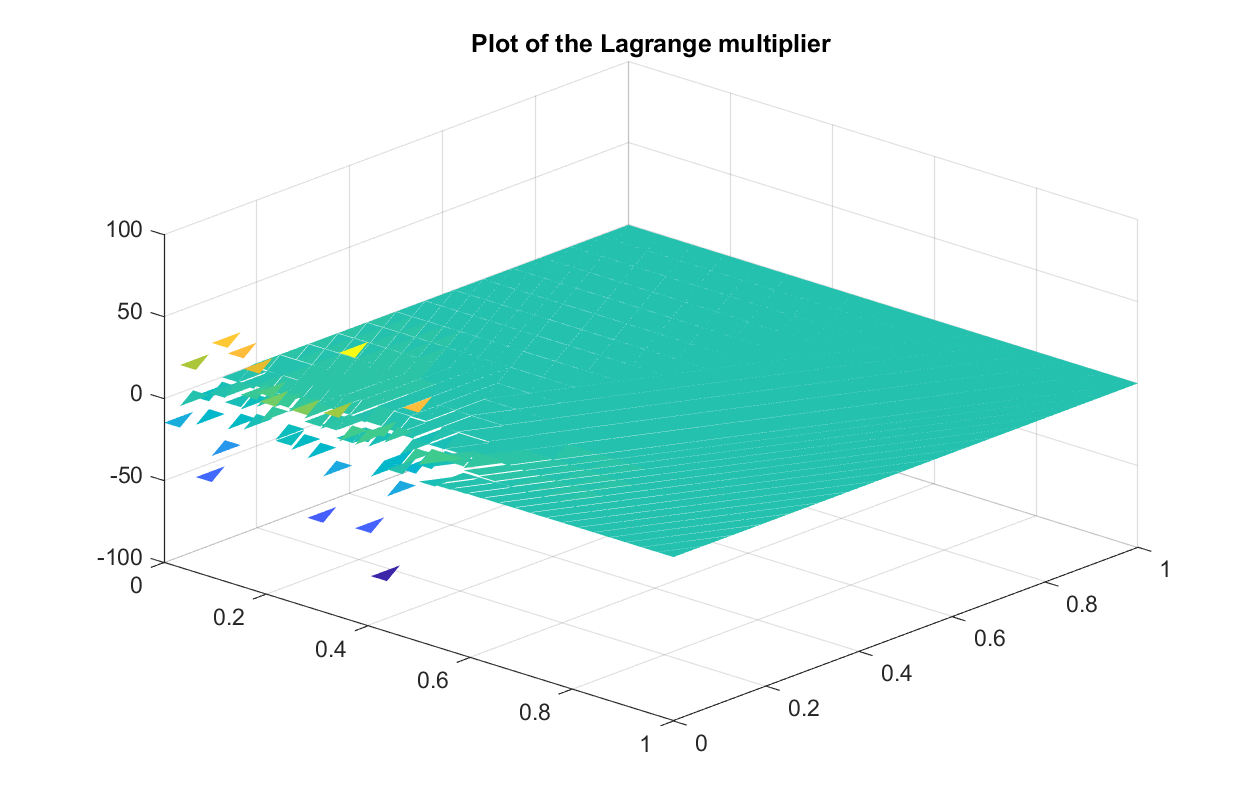}
}
\subfigure[Error function]{
\label{Fig.sub.5.3p}
\includegraphics [width=0.3\textwidth]{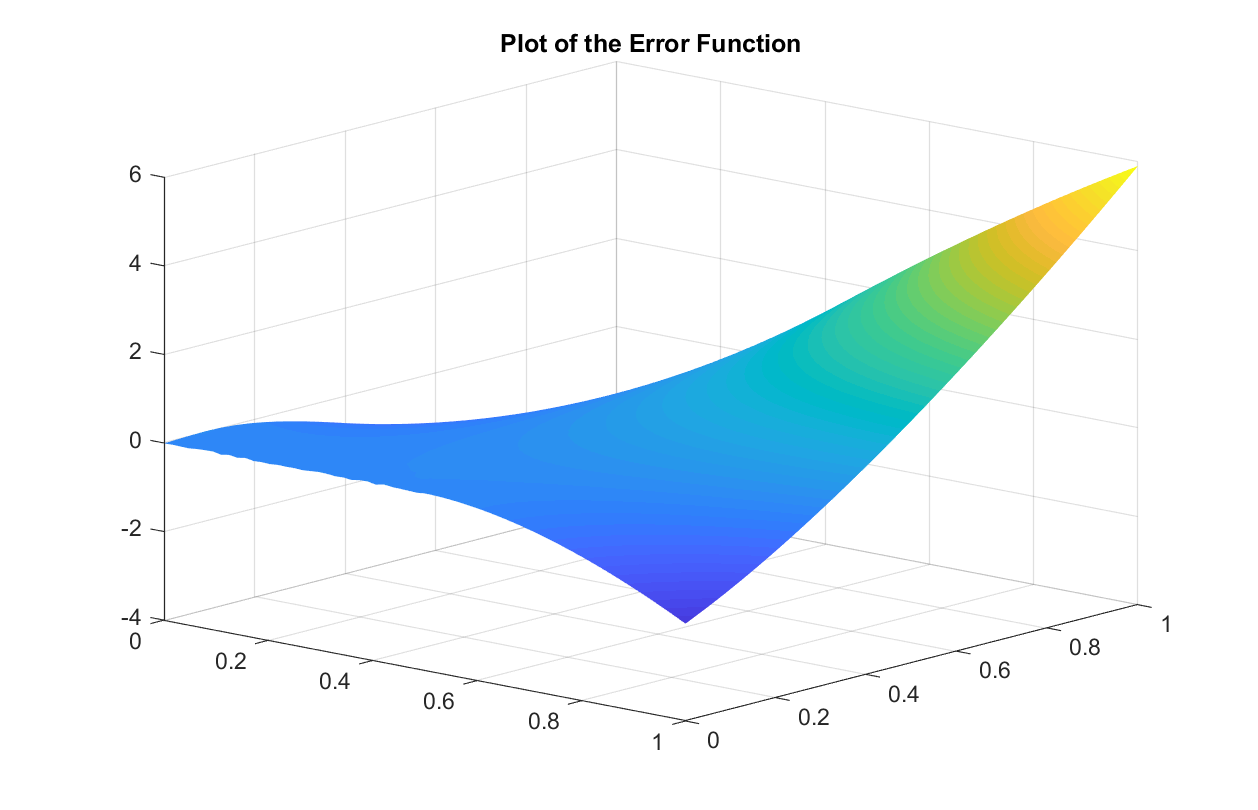}
}
\caption{Numerical results with a random perturbation by $0.05*(0.5-Rand)$ of the Cauchy data.}
\label{fig:mesh4}
\end{figure}

\end{document}